\documentclass[preprint,12pt]{elsarticle}

 \usepackage{epsfig}

\usepackage{amssymb}
\usepackage{amsmath}

\newcommand{\tuple}[1]{\vec{#1}}
\newcommand {\indep}[3] {#2 ~\bot_{#1}~ #3}
\newcommand {\indepc}[2] {#1 ~\bot~ #2}

\def\dom{\mbox{Dom}}
\def\free{\mbox{Free}}
\def\rel{\mbox{Rel}}
\def\part{\mathcal{P}}
\def\arity{\mbox{Arity}}
\def\traclo{\mbox{TC}}

\newtheorem{theo}{Theorem}[section]

\newdefinition{defin}[theo]{Definition}
\newtheorem{lemma}[theo]{Lemma}
\newtheorem{propo}[theo]{Proposition}
\newtheorem{coro}[theo]{Corollary}

\newproof{proof}{Proof}

\journal{arXiv}

\begin{document}

\begin{frontmatter}



\title{Inclusion and Exclusion Dependencies in Team Semantics\\\large{On some logics of imperfect information}}


\author{Pietro Galliani}

\ead{pgallian@gmail.com}
\address{Faculteit der Natuurwetenschappen, Wiskunde en Informatica\\
Institute for Logic, Language and Computation\\
Universiteit van Amsterdam\\
P.O. Box 94242, 1090 GE AMSTERDAM, The Netherlands\\
Phone: +31 020 525 8260\\
Fax (ILLC): +31 20 525 5206}

\begin{abstract}
We introduce some new logics of imperfect information by adding atomic formulas corresponding to \emph{inclusion} and \emph{exclusion} dependencies to the language of first order logic. The properties of these logics and their relationships with other logics of imperfect information are then studied. Furthermore, a game theoretic semantics for these logics is developed. As a corollary of these results, we characterize the expressive power of independence logic, thus answering an open problem posed in (Gr\"adel and V\"a\"an\"anen, 2010).
\end{abstract}


\end{frontmatter}


\section{Introduction}
The notions of dependence and independence are among the most fundamental ones considered in logic, in mathematics, and in many of their applications. For example, one of the main aspects in which modern predicate logic can be thought of as superior to medieval term logic is that the former allows for quantifier alternation, and hence can express certain complex patterns of dependence and independence between variables that the latter cannot easily represent. A fairly standard example of this can be seen in the formal representations of the notions of \emph{continuity} and \emph{uniform continuity}: in the language of first order logic, the former property can be expressed as $\forall x (\forall \epsilon > 0)(\exists \delta > 0) \forall x' ( |x - x'| < \delta \rightarrow |f(x) - f(x')| < \epsilon)$, while the latter can be expressed as $(\forall \epsilon > 0)(\exists \delta > 0) \forall x \forall x' (|x - x'| < \delta \rightarrow |f(x) - f(x')| < \epsilon)$. The difference between these two expressions should be clear: in the first one, the value of the variable $\delta$ is a function of the values of the variables $x$ and $\epsilon$, while in the second one it is a function of the value of the variable $\epsilon$ alone. This very notion of functional dependence also occurs, at first sight rather independently, as one of the fundamental concepts of Database Theory, and in that context it proved itself to be highly useful both for the specification and study of \emph{normal forms} and for that of \emph{constraints} over databases.\footnote{We will not discuss these issues in any detail in this work; for a handy reference, we suggest \cite{date03} or any other database theory textbook.}\\

\emph{Logics of imperfect information} are a family of logical formalisms whose development arose from the observation that not all possible patterns of dependence and independence between variables may be represented in first order logic. Among these logics, \emph{dependence logic} \cite{vaananen07} is perhaps the one most suited for the analysis of the notion of dependence itself, since it isolates it by means of \emph{dependence atoms} which correspond, in a very exact sense, to functional dependencies of the exact kind studied in Database Theory. The properties of this logic, and of a number of variants and generalizations thereof, have been the object of much research in recent years, and we cannot hope to give here an exhaustive summary of the known results. We will content ourselves, therefore, to recall (in Subsection \ref{subsect:deplog}) the ones that will be of particular interest for the rest of this work.\\

\emph{Independence logic} \cite{gradel10} is a recent variant of dependence logic. In this new logic, the fundamental concept that is being added to the first order language is not \emph{functional dependence}, as for the case of dependence logic proper, but \emph{informational independence}: as we will see, this is achieved by considering \emph{independence atoms} $\indep{x}{y}{z}$, whose informal meaning corresponds to the statement ``for any fixed value of $x$, the sets of the possible values for $y$ and $z$ are independent''. Just as dependence logic allows us to reason about the properties of functional dependence, independence logic does the same for this notion. Much is not known at the moment about independence logic; in particular, one open problem mentioned in \cite{gradel10} concerns the \emph{expressive power} of this formalism over open formulas. As we will see, a formula in a logic of imperfect information defines, for any suitable model $M$, the family of its \emph{trumps}, that is, the family of all sets of assignments (\emph{teams}, in the usual terminology of dependence logic) which satisfy the formula. This differs from the case of first order logic, in which formulas satisfy or do not satisfy single assignments, and the intuitive reason for this should be understandable: asking whether a statement such as ``the values of the variables $x$ and $y$ are independent'' holds with respect of a single variable assignment is meaningless, since such an assertion can be only interpreted with respect to a family of \emph{possible assignments}. A natural question is then \emph{which} families of sets of possible variable assignments may be represented in terms of independence logic formulas.\footnote{The analogous question for dependence logic was answered in \cite{kontinenv09}, and we will report that answer as Theorem \ref{theo:DLform} of the present work.} An upper bound for the answer is in \cite{gradel10} already: all classes of sets of assignments which are definable in independence logic correspond to second order relations which are expressible in existential second order logic. In this work, we will show that this is also a \emph{lower bound}: a class of sets of assignments is definable in independence logic if and only if it is expressible in existential second order logic. This result, which we will prove as Corollary \ref{coro:ILform}, implies that independence logic is not merely a formalism obtained by adding an arbitrary, although reasonable-looking, new kind of atomic formula to the first order language. It -- and any other formalism equivalent to it -- is instead a natural upper bound for a general family of logics of imperfect information: in particular, if over finite models an arbitrary logic of imperfect information characterizes only teams which are in NP then, by Fagin's theorem \cite{fagin74}, this logic is (again, over finite models) equivalent to some fragment of independence logic. \\

The way in which we reach this result is also perhaps of some interest. Even though functional dependence and informational independence are certainly very important notions, they are by no means the only ones of their kind that are of some relevance. In the field of database theory, a great variety of other constraints over relations\footnote{Such constraints are usually called \emph{dependencies}, for historical reasons; but they need not correspond to anything resembling the informal idea of dependency.} has indeed been studied. Two of the simplest such constraints are \emph{inclusion dependencies} and \emph{exclusion dependencies}, whose definitions and basic properties we will recall in Subsection \ref{subsect:incexc}; then, in Subsections \ref{subsect:inclog} and \ref{subsect:exclog}, we will develop and study the corresponding logics.\footnote{Subsection \ref{subsect:equilog} briefly considers the case of \emph{equiextension dependencies} and shows that, for our purposes, they are equivalent to inclusion dependencies.} As we will see, ``exclusion logic'' is equivalent, in a strong sense, to dependence logic, while ``inclusion logic'' is properly contained in independence logic but incomparable with dependence logic. Then, in Subsection \ref{subsect:ielogic}, we will consider \emph{inclusion/exclusion logic}, that is, the logic obtained by adding atoms for inclusion \emph{and} exclusion dependencies to the language of first order logic, and prove that it is equivalent to independence logic.\\

Section \ref{sect:gamesem} develops a \emph{game theoretic semantics} for inclusion/exclusion logic. A game-theoretic semantics assigns truth values to expressions according to the properties of certain \emph{semantic games} (often, but not always, in terms of the existence of \emph{winning strategies} for these games). Historically, the first semantics for logics of imperfect information were of this kind; and even though, for many purposes, team semantics is a more useful and clearer formalism, we will see that studying the relationship between game semantics and team semantics allows us to better understand certain properties of the semantic rules for disjunction and existential quantification. Then, in Section \ref{sect:defin}, we examine the classes of teams definable by inclusion/exclusion logic formulas (or equivalently, by independence logic formulas), and we prove that these are precisely the ones corresponding to second order relations definable in existential second order logic. \\

Finally, in the last section we show that, as a consequence of this, some of the most general forms of dependency studied in database theory are expressible in independence logic. This, in the opinion of the author, suggests that logics of imperfect information (and, in particular, independence logic) may constitute an useful theoretical framework for the study of such dependencies and their properties. 
\section{Dependence and independence logic}
In this section, we will recall a small number of known results about dependence and independence logic. Some of the basic definitions of these logics will be left unsaid, as they will be later recovered in a slightly more general setting in Subsection \ref{subsect:laxstrict}. This section and that subsection, taken together, can be seen as a very quick crash course on the field of logics of imperfect information; the reader who is already familiar with such logics can probably skim through most of it, paying however some attention to the discussion of independence logic of Subsection \ref{subsect:indlog}, the alternative semantic rules of Definition \ref{def:altsem} and the subsequent discussion.
\subsection{Dependence logic}
\label{subsect:deplog}
Dependence logic \cite{vaananen07} is, together with IF logic (\cite{hintikka96}, \cite{tulenheimo09}), one of the most widely studied logics of imperfect information.  In brief, it can be described as the extension of first order logic obtained by adding \emph{dependence atoms} $=\!\!(t_1 \ldots t_n)$ to its language, with the informal meaning of ``The value of the term $t_n$ is functionally determined by the values of the terms $t_1 \ldots t_{n-1}$''.\\

This allows us to express patterns of dependence and independence between variables which are not expressible in first order logic: for example, in the formula $\forall x \exists y \forall z \exists w (=\!\!(z, w) \wedge \phi(x,y,z,w))$ the choice of the value for the variable $w$ depends only on the value of the variable $w$, and not from the values of the variables $x$ and $y$ - or, in other words, this expression is equivalent to the branching quantifier (\cite{henkin61}) sentence
\[
\left(\begin{array}{c c}
	\forall x & \exists y \\
	\forall z & \exists w 
\end{array}\right)
 \phi(x, y, z, w)
\]
and the corresponding Skolem normal form is $\exists f \exists g \forall x \forall z \phi(x, f(x), z, g(z))$.

The idea of allowing more general patterns of dependence and independence between quantifiers than the ones permitted in first order logic was, historically, the main reason for the development of logics of imperfect information: in particular, \cite{hintikka96} argues that the restriction on these patterns forced by first order logic has little justification, and that hence logics of imperfect information are a more adequate formalism for reasoning about the foundations of mathematics. 

No such claim will be made or discussed in this work. But in any case, the idea of allowing more general patterns of dependence and independence between quantifiers seems a very natural one. In IF logic, the notion of dependence is, however, inherently connected with the notion of quantification: for example, the above expression would be written in it as $\forall x \exists y \forall z (\exists z / x, y) \phi(x, y, z, w)$, where $(\exists z / x, y)$ is to be read as ``there exists a $z$, independent from $x$ and $y$, such that \ldots''. Dependence logic and its variants, instead, prefer to separate the notion of dependency from the notion of quantification: in this second group of logics of imperfect information, dependence patterns between quantifiers are exactly as first order logic and our linguistic intuitions would suggest, but dependence atoms may be used to specify that the value of a certain variable (or, in general, of a certain term) must be a function of certain other values. This corresponds precisely to the notion of \emph{functional dependence} which is one of the central tools of Database Theory; and indeed, as we will recall later in this work, the satisfaction conditions for these atoms are in a very precise relationship with the formal definition of functional dependence. 

This, at least in the opinion of the author, makes dependence logic an eminently suitable formalism for the study of the notion of functional dependence and of its properties; and as we will see, one of the main themes of the present work will consist in the development and study of formalisms which have a similar sort of relationship with other notions of dependency.\\

We will later recall the full definition of the \emph{team semantics} of dependence logic, an adaptation of Hodges' compositional semantics for IF-logic (\cite{hodges97}) and one of the three equivalent semantics for dependence logic described in \cite{vaananen07}.\footnote{The readers interested in a more thorough explanation of the team semantics and of the two game theoretic semantics for dependence logic are referred to \cite{vaananen07} itself.} It is worth noting already here, though, that the key difference between Hodges semantics and the usual Tarskian semantics is that in the former semantics the satisfaction relation $\models$ associates to every first order model\footnote{In all this paper, I will assume that first order models have at least two elements in their domain.} $M$ and formula $\phi$ a set of \emph{teams}, that is, a set of sets of assignments, instead of just a set of assignments as in the latter one.\\

As discussed in \cite{hodges07}, the fundamental intuition behind Hodges' semantics is that a team is a representation of an \emph{information state} of some agent: given a model $M$, a team $X$ and a suitable formula $\phi$, the expression
\[
	M \models_X \phi
\]
asserts that, from the information that the ``true'' assignment $s$ belongs to the team $X$, it is possible to infer that $\phi$ holds - or, in game-theoretic terms, that the Verifier has a strategy $\tau$ which is winning for all plays of the game $G(\phi)$ which start from any assignment $s \in X$. 

The satisfaction conditions for the dependence atom is then given by the following semantic rule \textbf{TS-dep}:
\begin{defin}[Dependence atoms]
	\label{def:dep}
	Let $M$ be a first order model, let $X$ be a team over it, let $n \in \mathbb N$, and let $t_1 \ldots t_n$ be terms over the signature of $M$ and with variables in $\dom(X)$. Then 
\begin{description}
\item[TS-dep:] $M \models_X =\!\!(t_1 \ldots t_n)$ if and only if, for all $s, s' \in X$ such that  $t_i \langle s \rangle = t_i\langle s'\rangle \mbox{ for } i = 1 \ldots n-1$, $t_n\langle s \rangle = t_n\langle s'\rangle$.
\end{description}
\end{defin}

This rule corresponds closely to the definition of \emph{functional dependency} commonly used in Database Theory (\cite{codd72}): more precisely, if $X(t_1 \ldots t_n)$ is the relation $\{(t_1\langle s\rangle, \ldots, t_n\langle s \rangle) : s \in X\}$ then
\[
M \models_X =\!\!(t_1 \ldots t_n) \Leftrightarrow X(t_1 \ldots t_n) \models \{t_1 \ldots t_{n-1}\} \rightarrow t_n
\]
where the right hand expression states that, in the relation $X(t_1 \ldots t_n)$, the value of the last term $t_n$ is a function of the values of $t_1 \ldots t_{n-1}$. 

Another formulation of the truth condition of a dependence atom\\ $=\!\!(t_1 \ldots t_n)$, easily seen to be equivalent to this one, is the following: a team $X$ satisfies such an atom if and only if a rational agent $\alpha$, whose beliefs about the identity of the ``true'' assignment $s$ are described by $X$, would be capable of inferring the value of $t_n$ from the values of $t_1 \ldots t_{n-1}$.\footnote{Decomposing the notion further, this is equivalent to stating that if the values of $t_1 \ldots t_{n-1}$ for the true assignment $s \in X$ were announced to the agent then he or she would also learn the value of $t_{n}$. The properties of this sort of announcement operators for dependence logic are discussed in \cite{galliani10}.} A special case of dependence atom, useful to consider in order to clarify our intuitions, is constituted by \emph{constancy atoms} $=\!\!(t)$: applying the above definitions, we can observe that $M \models_X =\!\!(t)$ if and only if the value $t\langle s\rangle$ is the same for all assignments $s \in X$ - or, using the agent metaphor, if and only if an agent $\alpha$ as above \emph{knows} the value of $t$.\footnote{The existence of a relation between these notions and the ones studied in the field of epistemic modal logic is clear, but to the knowledge of the author the matter has not yet been explored in full detail. See \cite{vanbenthem06} for some intriguing reflections about this topic.}\\

The following known results will be of some use for the rest of this work:
\begin{theo}[Locality \cite{vaananen07}]
	\label{theo:DLloc}
	Let $M$ be a first order model and let $\phi$ be a dependence logic formula over the signature of $M$ with free variables in $\tuple v$. Then, for all teams $X$ with domain $\tuple w \supseteq \tuple v$, if $X'$ is the restriction of $X$ to $\tuple v$ then 
	\[
		M \models_X \phi \Leftrightarrow M \models_{X'} \phi.
	\]
\end{theo}
As an aside, it is worth pointing out that the above property does not hold for most variants of $IF$-logic: for example, if $\dom(M) = \{0,1\}$ and $X = \{(x:0, y:0), (x:1, y:1)\}$ it is easy to see that $M \models_X (\exists z / y) z = y$, even though for the restriction $X'$ of $X$ to $\free((\exists z/y) z=y) = \{y\}$ we have that $M \not \models_{X'} (\exists z/y) z=y$.\footnote{This is a typical example of \emph{signalling} (\cite{hintikka96}, \cite{janssen06}), one of the most peculiar and, perhaps, problematic aspects of $IF$-logic.}

\begin{theo}[Downwards Closure Property \cite{vaananen07}]
	\label{theo:DLdc}
	Let $M$ be a model, let $\phi$ be a dependence logic formula over the signature of $M$, and let $X$ be a team over $M$ with domain $\tuple v \supseteq \free(\phi)$ such that $M \models_X \phi$. Then, for all $X' \subseteq X$, 
	\[
		M \models_{X'} \phi. 
	\]
\end{theo}

\begin{theo}[Dependence logic sentences and $\Sigma_1^1$ \cite{vaananen07}]
	\label{theo:DLsent}
	For every dependence logic sentence $\phi$, there exists a $\Sigma_1^1$ sentence $\Phi$ such that 
	\[
		M \models_{\{\emptyset\}} \phi \Leftrightarrow M \models \Phi.
	\]

	Conversely, for every $\Sigma_1^1$ sentence $\Phi$ there exists a dependence logic sentence $\phi$ such that the above holds. 
\end{theo}

\begin{theo}[Dependence logic formulas and $\Sigma_1^1$ \cite{kontinenv09}] 
	\label{theo:DLform}
	For every dependence logic formula $\phi$ and every tuple of variables $\tuple x \supseteq \free(\phi)$ there exists a $\Sigma_1^1$ sentence $\Phi(R)$, where $R$ is a $|\tuple x|$-ary relation which occurs only negatively in $\Phi$, such that 
	\[
		M \models_X \phi \Leftrightarrow M \models \Phi(\rel(X))
	\]
	for all teams $X$ with domain $\tuple x$.\footnote{Here $\rel(X)$ is the relation corresponding to the team $X$, as in Definition \ref{defin:relteams}.} 

	Conversely, for all such $\Sigma_1^1$ sentences there exists a dependence logic formula $\phi$ such that the above holds with respect to all \emph{nonempty} teams $X$. 
\end{theo}
\subsection{Independence logic} 
\label{subsect:indlog}
Independence logic \cite{gradel10} is a recently developed logic which substitutes the dependence atoms of dependence logic with \emph{independence atoms} $\indep{\tuple t_1}{\tuple t_2}{\tuple t_3}$, where $\tuple t_1 \ldots \tuple t_3$ are tuples of terms (not necessarily of the same length). 

The intuitive meaning of such an atom is that the values of the tuples $\tuple t_2$ and $\tuple t_3$ are informationally independent for any fixed value of $\tuple t_1$; or, in other words, that all information about the value of $\tuple t_3$ that can be possibly inferred from the values of $\tuple t_1$ and $\tuple t_2$ can be already inferred from the value of $\tuple t_1$ alone. \\

More formally, the definition of the team semantics for the independence atom is as follows:
\begin{defin}[Independence atoms]
	\label{def:indep}
	Let $M$ be a first order model, let $X$ be a team over it and let $\tuple t_1, \tuple t_2$ and $\tuple t_3$ be three finite tuples of terms (not necessarily of the same length) over the signature of $M$ and with variables in $\dom(X)$. Then 
\begin{description}
	\item[TS-indep:] $M \models_X \indep{\tuple t_1}{\tuple t_2}{\tuple t_3}$ if and only if for all $s, s' \in X$ with $\tuple t_1\langle s\rangle = \tuple t_1\langle s\rangle$ there exists a $s'' \in X$ such that $\tuple t_1\langle s''\rangle \tuple t_2\langle s''\rangle = \tuple t_1 \langle s\rangle \tuple t_2\langle s\rangle$  and $\tuple t_1 \langle s''\rangle \tuple t_3\langle s''\rangle = \tuple t_1\langle s'\rangle \tuple t_3\langle s'\rangle$.
\end{description}
\end{defin}

We refer to \cite{gradel10} for a discussion of this interesting class of atomic formulas and of the resulting logic. Here we only mention a few results, found in that paper, which will be useful for the rest of this work:\footnote{Another interesting result about independence logic, pointed out by Fredrik Engstr\"om in \cite{engstrom10}, is that the semantic rule for independence atoms corresponds to that of \emph{embedded multivalued dependencies}, in the same sense in which the one for dependence atoms corresponds to \emph{functional} ones.}
\begin{theo}
	\label{theo:DL2IL}
	Dependence atoms are expressible in terms of independence atoms: more precisely, for all suitable models $M$, teams $X$ and terms $t_1 \ldots t_n$ 
	\[
		M \models_X =\!\!(t_1 \ldots t_n) \Leftrightarrow M \models_X \indep{t_1 \ldots t_{n-1}}{t_n}{t_n}.
	\]
\end{theo}
\begin{theo}
	\label{theo:ILsent}
	Independence logic is equivalent to $\Sigma_1^1$ (and therefore, by Theorem \ref{theo:DLsent}, to dependence logic) over sentences: in other words, for every sentence $\phi$ of independence logic there exists a sentence $\Phi$ of existential second order logic such that 
	\[
	M \models_{\{\emptyset\}} \phi \Leftrightarrow M \models \Phi.
	\]
	and for every such $\Phi$ there exists a $\phi$ such that the above holds. 
\end{theo}

There is no analogue of Theorem \ref{theo:DLdc} for independence logic, however, as the classes of teams corresponding to independence atoms are not necessarily downwards closed: for example, according Definition \ref{def:indep} the formula $\indep{\emptyset}{x}{y}$ holds in the team 
\[
	\{(x:0, y:0), (x:0, y:1), (x:1,y:0), (x:1, y:1)\}
\]
but not in its subteam $\{(x:0,y:0), (x:1,y:1)\}$.\\

The problem of of finding a characterization similar to that of Theorem \ref{theo:DLform} for the classes of teams definable by formulas of independence logic was left open by Gr\"adel and V\"a\"an\"anen, who concluded their paper by stating that (\cite{gradel10})
\begin{quote}
\emph{The main open question raised by the above discussion is the following, formulated for finite structures:}\\

\textbf{Open Problem:} \emph{Characterize the NP properties of teams that correspond to formulas of independence logic.}
\end{quote}

In this paper, an answer to this question will be given, as a corollary of an analogous result for a new logic of imperfect information.
\section{Team semantics}
In this section, we will introduce some of the main concepts that we will need for the rest of this work and then we will test them on a relatively simple case. Subsection \ref{subsect:laxstrict} contains the basic definitions of team semantics, following for the most part the treatment of \cite{vaananen07}; and furthermore, in this subsection we introduce two variant rules for disjunction and existential quantification which, as we will later see, will be of significant relevance. Then, in Subsection \ref{subsect:constancy}, we will begin our investigations by examining \emph{constancy logic}, that is, the fragment of dependence logic obtained by adding constancy atoms to the language of first order logic. The main result of that subsection will be a proof that constancy logic is expressively equivalent to first order logic over sentences, and, hence, that it is strictly less expressive than the full dependence logic. This particular consequence is a special case of the far-reaching \emph{hierarchy theorem} of \cite{durand11}, which fully characterizes the expressive powers of certain fragments of dependence logic. 
\subsection{First order (team) logic, in two flavors}
\label{subsect:laxstrict}
In this subsection, we will present and briefly discuss the team semantics for first order logic, laying the groundwork for reasoning about its extensions while avoiding, as far as we are able to do so, all forms of semantical ambiguity. \\

As we will see, some special care is required here, since certain rules which are equivalent with respect to dependence logic proper will not be so with respect to these new logics. As it often is the case for logics of imperfect information, the game theoretic approach to semantics (which we will discuss in Section \ref{sect:gamesem}) will be of support and clarification for our intuitions concerning the intended interpretations of operators.\\

But let us begin by recalling some basic definitions from \cite{vaananen07}:
\begin{defin}[Team]
	Let $M$ be a first order model, and let $\tuple v$ be a tuple of variables.\footnote{Or, equivalently, a \emph{set} of variables; but having a fixed ordering of the variables as part of the definition of team will simplify the definition of the correspondence between teams and relations. With an abuse of notation, we will identify this tuple of variables with the underlying set whenever it is expedient to do so.} Then a \emph{team} $X$ for $M$ with \emph{domain} $\tuple v$ is simply a set of assignments with domain $\tuple v$ over $M$.
\end{defin}
\begin{defin}[From teams to relations]
	\label{defin:relteams}
	Let $M$ be a first order model, $X$ be a team for $M$ with domain $\tuple v$, and let $\tuple t = t_1 \ldots t_k$ be a tuple of terms with variables in $\tuple v$. Then we write $X(\tuple t)$ for the relation
	\[
		X(\tuple t) = \{(t_1\langle s\rangle \ldots t_k\langle s \rangle) : s \in X\}.
	\]
	Furthermore, if $\tuple w$ is contained in $\tuple v$ we will write $\rel_{\tuple w}(X)$ for $X(\tuple w)$; and, finally, if $\dom(X) = \tuple v$ we will write $\rel(X)$ for $\rel_{\tuple v}(X)$. 
\end{defin}
\begin{defin}[Team restrictions]
	Let $X$ be any team in any model, and let $V$ be a set of variables contained in $\dom(X)$. Then 
	\[
		X_{\upharpoonright V} = \{s_{\upharpoonright V} : s \in X\}
	\]
	where $s_{\upharpoonright V}$ is the restriction of $s$ to $V$, that is, the only assignment $s'$ with domain $V$ such that $s'(v) = s(v)$ for all $v \in V$.
\end{defin}
The team semantics for the first order fragment of dependence logic is then defined as follows:
\begin{defin}[Team semantics for first order logic (\cite{hodges97}, \cite{vaananen07})]
	\label{def:team_fol}
	Let $M$ be a first order model, let $\phi$ be a first order formula in negation normal form\footnote{Since the negation is not a semantic operation in dependence logic (\cite{burgess03}, \cite{kontinenv10}), it is useful to assume that all formulas are in negation normal form. It is of course possible to adapt these definitions to formulas not in negation normal form, but in order to do so for the cases of dependence or independence logic it would be necessary to define two distinct relationships $\models^+$ and $\models^-$, as in \cite{vaananen07}. Since, for the purposes of this work, this would offer no significant advantage and would complicate the definitions, it was chosen to avoid the issue by requiring all formulas to be in negation normal form instead.} and let $X$ be a team over $M$ with domain $\tuple v \supseteq \free(\phi)$. Then 
	\begin{description}
			\item[TS-atom:] If $\phi$ is a first order literal, $M \models_X \phi$ if and only if, for all assignments $s \in X$, $M \models_s \phi$ in the usual first order sense; 
			\item[TS-$\vee_L$:] If $\phi$ is $\psi \vee \theta$, $M \models_X \phi$ if and only if there exist two teams $Y$ and $Z$ such that $X = Y \cup Z$, $M \models_Y \psi$ and $M \models_Z \theta$; 
			\item[TS-$\wedge$:] If $\phi$ is $\psi \wedge \theta$, $M \models_X \phi$ if and only if $M \models_X \psi$ and $M \models_X \theta$; 
			\item[TS-$\exists_S$:] If $\phi$ is $\exists x \psi$, $M \models_X \phi$ if and only if there exists a function $F: X \rightarrow \dom(M)$ such that $M \models_{X[F/x]}\psi$, where 
				\[
					X[F/x] = \{s[F(s)/x] : s \in X\};\footnote{Sometimes, we will write $X[F_1 F_2 \ldots F_n / x_1 \ldots x_n]$, or even $X[\tuple F/\tuple x]$, as a shorthand for $X[F_1/x_1][F_2/x_2]\ldots[F_n/x_n]$.}
				\]
			\item[TS-$\forall$:] If $\phi$ is $\forall x \psi$, $M \models_X \phi$ if and only if $M \models_{X[M/x]} \psi$, where 
				\[
					X[M/x] = \{s[m/x] : s \in X\}.\footnote{Sometimes, we will write $X[M/x_1 x_2 \ldots x_n]$, or even $X[M/\tuple x]$, as a shorthand for $X[M/x_1][M/x_2]\ldots[M/x_n]$.}
				\]
	\end{description}
\end{defin}
Over singleton teams, this semantics coincides with the usual one for first order logic: 
\begin{propo}[\cite{vaananen07}]
	\label{propo:FO_Team2Tarski}
	Let $M$ be a first order model, let $\phi$ be a first order formula in negation normal form over the signature of $M$, and let $s$ be an assignment with $\dom(s) \supseteq \free(\phi)$. Then $M \models_{\{s\}} \phi$ if and only if $M \models_{s} \phi$ with respect to the usual Tarski semantics for first order logic. 
\end{propo}
Furthermore, as the following proposition illustrates, the team semantics of first order logic is compatible with the intuition, discussed before, that teams represent states of knowledge: 
\begin{propo}[\cite{vaananen07}]
	\label{propo:FOflat}
	Let $M$ be a first order model, let $\phi$ be a first order formula in negation normal form over the signature of $M$, and let $X$ be a team with $\dom(X) \supseteq \free(\phi)$. Then $M \models_X \phi$ if and only if, for all assignments $s \in X$, $M \models_{\{s\}} \phi$.\footnote{In other words, first order formulas are \emph{flat} in the sense of \cite{vaananen07}.}  
\end{propo}

On the other hand, these two proposition also show that, for first order logic, all the above machinery is quite unnecessary. We have no need of carrying around such complex objects as teams, since we can consider any assignment in a team individually!\\

Things, however, change if we add dependence atoms $=\!\!(t_1 \ldots t_n)$ to our language, with the semantics of rule \textbf{TS-dep} (Definition \ref{def:dep} here). In the resulting formalism, which is precisely \emph{dependence logic} as defined in \cite{vaananen07}, not all satisfaction conditions over teams can be reduced to satisfaction conditions over assignments: for example, a ``constancy atom'' $=\!\!(x)$ holds in a team $X$ if and only if $s(x) = s'(x)$ for all $s, s' \in X$, and verifying this condition clearly requires to check \emph{pairs} of assignments at least!\footnote{That is, all constancy atoms - and, more in general, all dependence atoms - are $2$\emph{-coherent} but not $1$\emph{-coherent} in the sense of \cite{kontinen_ja10}.}\\

When studying variants of dependence logic, similarly, it is necessary to keep in mind that semantic rules which are equivalent with respect to dependence logic proper may not be equivalent with respect to these new formalisms. In particular, two alternative definitions of disjunction and existential quantification exist which are of special interest for this work's purposes:\footnote{The rule \textbf{TS}-$\exists_L$ is also discussed in \cite{engstrom10}, in which it is shown that it arises naturally from treating the existential quantifier as a \emph{generalized quantifier} (\cite{mostowski57}, \cite{lindstrom66}) for dependence logic.}
\begin{defin}[Alternative rules for disjunctions and existentials]
\label{def:altsem}
	Let $M$, $X$, $\phi$, $\psi$ and $\theta$ be as usual. Then 
\begin{description}
	\item[TS-$\vee_S$:] If $\phi$ is $\psi \vee \theta$, $M \models_X \phi$ if and only if there exist two teams $Y$ and $Z$ such that $X = Y \cup Z$, $Y \cap Z = \emptyset$, $M \models_Y \psi$ and $M \models_Z \theta$; 
	\item[TS-$\exists_L$:] If $\phi$ is $\exists x \psi$, $M \models_X \phi$ if and only if there exists a function $H: X \rightarrow \part(\dom(M))\backslash \emptyset$ such that $M \models_{X[H/x]}\psi$, where 
				\[
					X[H/x] = \{s[m/x] : s \in X, m \in H(s)\}.
				\]
\end{description}
\end{defin}
The subscripts of $\cdot_S$ and $\cdot_L$ of these rules and of the corresponding ones of Definition \ref{def:team_fol} allow us to discriminate between the \emph{lax} operators $\vee_L$ and $\exists_L$ and the \emph{strict} ones $\vee_S$ and $\exists_S$. This distinction will be formally justified in Section \ref{sect:gamesem}, and in particular by Theorems \ref{theo:game_team_lax} and \ref{theo:game_team_strict}; but even at a glance, this grouping of the rules is justified by the fact that \textbf{TS-$\vee_S$} and \textbf{TS-$\exists_S$} appear to be stronger conditions than \textbf{TS-$\vee_L$} and \textbf{TS-$\exists_L$}. We can then define two alternative semantics for first order logic (and for its extensions, of course) as follows:
\begin{defin}[Lax semantics]
	\label{def:team_fol_lax}
	The relation $M \models_X^L \phi$, where $M$ ranges over all first order models, $X$ ranges over all teams and $\phi$ ranges over all formulas with free variables in $\dom(X)$, is defined as the relation $M \models_X \phi$ of Definition \ref{def:team_fol} (with additional rules for further atomic formulas as required), but substituting Rule \textbf{TS-$\exists_S$} with Rule \textbf{TS-$\exists_L$}.
\end{defin}
\begin{defin}[Strict semantics]
	\label{def:team_fol_strict}
	The relation $M \models_X^S \phi$, where $M$ ranges over all first order models, $X$ ranges over all teams and $\phi$ ranges over all formulas with free variables in $\dom(X)$, is defined as the relation $M \models_X \phi$ of Definition \ref{def:team_fol} (with additional rules for further atomic formulas as required), but substituting Rule \textbf{TS-$\vee_L$} with Rule \textbf{TS-$\vee_S$}.
\end{defin}

For the cases of first order and dependence logic, the lax and strict semantics are equivalent:
\begin{propo}
	\label{propo:laxeqstrict}
	Let $\phi$ be any formula of dependence logic. Then
	\[
		M \models_X^S \phi \Leftrightarrow M \models^L_X \phi
	\]
	for all suitable models $M$ and teams $\phi$.
\end{propo}
\begin{proof}
This is easily verified by structural induction over $\phi$, using the downwards closure property (Theorem \ref{theo:DLdc}) to take care of disjunctions and existentials (and, moreover, applying the Axiom of Choice for the case of existentials). We verify the case corresponding to existential quantifications, as an example: the one corresponding to disjunctions is similar but simpler, and the the others are trivial. \\

Suppose that $M \models_X^S \exists x \phi$: then, by rule \textbf{TS-$\exists_S$}, there exists a function $F: X \rightarrow \dom(M)$ such that $M \models_{X[F/x]}^S \phi$. Now define the function $H: X \rightarrow \part(\dom(M)) \backslash \{\emptyset\}$ so that, for all $s \in X$, $H(s) = \{F(s)\}$: then $X[H/x] = X[F/x]$, and therefore by induction hypothesis $M \models_{X[H/x]}^L \phi$, and hence by rule \textbf{TS-$\exists_L$} $M \models_X^L \exists x \phi$. Conversely, suppose that $M \models_X^L \exists x \phi$: then, by rule \textbf{TS-$\exists_L$}, there exists a function $H: X \rightarrow \part(\dom(M)) \backslash \{\emptyset\}$ such that $M \models_{X[H/x]}^L \phi$. Then, by the Axiom of Choice, there exists a \emph{choice function} $F: X \rightarrow \dom(X)$ such that, for all $s \in X$, $F(s) \in H(s)$; therefore, $X[F/x] \subseteq X[H/x]$ and, by downwards closure, $M \models_{X[F/x]}^L \phi$. But then by induction hypothesis $M \models_{X[F/x]}^S \phi$ and, by rule \textbf{TS-$\exists_L$}, $M \models_X^S \phi$. 
\end{proof}

As we will argue in Section \ref{subsect:inclog}, for the logics that we will study for which a difference exists between lax and strict semantics the former will be the most natural choice; therefore, from this point until the end of this work the symbol $\models$  written without superscripts will stand for the relation $\models^L$.
\subsection{Constancy logic}
\label{subsect:constancy}
In this section, we will present and examine a simple fragment of dependence logic. This fragment, which we will call \emph{constancy logic}, consists of all the formulas of dependence logic in which only dependence atoms of the form $=\!\!(t)$ occur; or, equivalently, it can be defined as the extension of (team) first order logic obtained by adding \emph{constancy atoms} to it, with the semantics given by the following definition:
\begin{defin}[Constancy atoms]
	\label{def:const}
	Let $M$ be a first order model, let $X$ be a team over it, and let $t$ be a term over the signature of $M$ and with variables in $\dom(X)$. Then 
\begin{description}
\item[TS-const:] $M \models_X =\!\!(t)$ if and only if, for all $s, s' \in X$, 
	$t\langle s\rangle = t\langle s'\rangle$.
\end{description}
\end{defin}
Clearly, constancy logic is contained in dependence logic. Furthermore, over open formulas it is more expressive than first order logic proper, since, as already mentioned, the constancy atom $=(x)$ is a counterexample to Proposition \ref{propo:FOflat}.\\

The question then arises whether constancy logic is properly contained in dependence logic, or if it coincides with it. This will be answered through the following results:
\begin{propo}
	\label{propo:const_out}
	Let $\phi$ be a constancy logic formula, let $z$ be a variable not occurring in $\phi$, and let $\phi'$ be obtained from $\phi$ by substituting one instance of $=\!\!(t)$ with the expression $z = t$. 

	Then 
	\[
		M \models_X \phi \Leftrightarrow M \models_X \exists z (=\!\!(z) \wedge \phi').
	\]
\end{propo}
\begin{proof}
	The proof is by induction on $\phi$.
	\begin{enumerate}
		\item If the expression $=\!\!(t)$ does not occur in $\phi$, then $\phi' = \phi$ and we trivially have that $\phi \equiv \exists z (=\!\!(z) \wedge \phi)$, as required.
		\item If $\phi$ is $=\!\!(t)$ itself then $\phi'$ is $z = t$, and 
			\begin{align*}
				& M \models_X \exists z (=\!\!(z) \wedge z = t) \Leftrightarrow 
				\exists m  \in \dom(M)\mbox{ s.t. } M \models_{X[m/z]} z = t \Leftrightarrow\\
				&\Leftrightarrow \exists m \in \dom(M) \mbox{ s.t. } t\langle s\rangle = m \mbox{ for all } s \in X \Leftrightarrow M \models_X =\!\!(t)
			\end{align*}
			as required, where we used $X[m/z]$ as a shorthand for $\{s(m/z) : s \in X\}$.
		\item If $\phi$ is $\psi_1 \vee \psi_2$, let us assume without loss of generality that the instance of $=\!\!(t)$ that we are considering is in $\psi_1$. Then $\psi'_2 = \psi_2$, and since $z$ does not occur in $\psi_2$
			\begin{align*}
			&M \models_X \exists z(=\!\!(z) \wedge (\psi'_1 \vee \psi_2)) \Leftrightarrow \exists m \mbox{ s.t. } M \models_{X[m/z]} \psi'_1 \vee \psi_2 \Leftrightarrow\\
			&\Leftrightarrow \exists m, X_1, X_2 \mbox{ s.t. } X_1 \cup X_2 = X, M \models_{X_1[m/z]} \psi'_1 \mbox{ and } M \models_{X_2[m/z]} \psi_2 \Leftrightarrow\\
			&\Leftrightarrow \exists m, X_1, X_2 \mbox{ s.t. } X_1 \cup X_2 = X, M \models_{X_1[m/z]} \psi'_1 \mbox{ and } M \models_{X_2} \psi_2 \Leftrightarrow\\
			&\Leftrightarrow X_1, X_2 \mbox{ s.t. } X_1 \cup X_2 = X, M \models_{X_1} \exists z (=\!\!(z) \wedge \psi'_1) \mbox{ and } M \models_{X_2} \psi_2 \Leftrightarrow\\
			&\Leftrightarrow X_1, X_2 \mbox{ s.t. } X_1 \cup X_2 = X, M \models_{X_1} \psi_1 \mbox{ and } M \models_{X_2} \psi_2 \Leftrightarrow\\
			&\Leftrightarrow M \models_{X} \psi_1 \vee \psi_2
		\end{align*}
		as required.
	\item If $\phi$ is $\psi_1 \wedge \psi_2$, let us assume again that the instance of $=\!\!(t)$ that we are considering is in $\psi_1$. Then $\psi_2' = \psi_2$, and 
		\begin{align*}
			&M \models_X \exists z(=\!\!(z) \wedge \psi'_1 \wedge \psi_2) \Leftrightarrow\\
			&\Leftrightarrow \exists m \mbox{ s.t. } M \models_{X[m/z]} \psi'_1 \mbox{ and } M \models_{X[m/z]} \psi_2 \Leftrightarrow\\
			&\Leftrightarrow M \models_X \exists z(=\!\!(z) \wedge \psi'_1) \mbox{ and } M \models_X \psi_2 \Leftrightarrow\\
			&\Leftrightarrow M \models_X \psi_1 \mbox{ and } M \models_X \psi_2 \Leftrightarrow\\
			&\Leftrightarrow M \models_X \psi_1 \wedge \psi_2.
		\end{align*}
	\item If $\phi$ is $\exists x \psi$, 
		\begin{align*}
			&M \models_X \exists z (=\!\!(z) \wedge \exists x \psi') \Leftrightarrow \\
			&\Leftrightarrow \exists m \mbox{ s.t. } M \models_{X[m/z]} \exists x \psi' \Leftrightarrow\\
			&\Leftrightarrow \exists m, \exists H: X[m/z] \rightarrow \part(\dom(M)) \backslash \{\emptyset\} \mbox{ s.t. } M \models_{X[m/z][H/x]} \psi' \Leftrightarrow\\
			&\Leftrightarrow \exists H': X \rightarrow \part(\dom(M)) \backslash \{\emptyset\}, \exists m \mbox{ s.t. } M \models_{X[H'/x][m/z]} \psi' \Leftrightarrow\\
			&\Leftrightarrow \exists H': X \rightarrow \part(\dom(M)) \backslash \{\emptyset\} \mbox{ s.t. } M \models_{X[H'/x]} \exists z (=\!\!(z) \wedge \psi') \Leftrightarrow\\
			&\Leftrightarrow \exists H': X \rightarrow \part(\dom(M)) \backslash \{\emptyset\}, \mbox{ s.t. } M \models_{X[H'/x]} \psi \Leftrightarrow\\
			&\Leftrightarrow M \models_X \exists x \psi.
		\end{align*}
	\item If $\phi$ is $\forall x \psi$, 
		\begin{align*}
			&M \models_X \exists z(=\!\!(z) \wedge \forall x \psi') \Leftrightarrow\\
			&\Leftrightarrow \exists m \mbox{ s.t. } M \models_{X[m/z]} \forall x \psi' \Leftrightarrow \\
			&\Leftrightarrow \exists m \mbox{ s.t. } M \models_{X[m/z][M/x]} \psi' \Leftrightarrow \\
			&\Leftrightarrow \exists m \mbox{ s.t. } M \models_{X[M/x][m/z]} \psi' \Leftrightarrow \\
			&\Leftrightarrow M \models_{X[M/x]} \exists z (=\!\!(z) \wedge \psi') \Leftrightarrow \\
			&\Leftrightarrow M \models_{X[M/x]} \psi \Leftrightarrow \\
			&\Leftrightarrow M \models_{X} \forall x \psi.
		\end{align*}
	\end{enumerate}
\end{proof}
As a corollary of this result, we get the following normal form theorem for constancy logic:\footnote{This normal form theorem is very similar to the one of dependence logic proper found in \cite{vaananen07}. See also \cite{durand11} for a similar, but not identical result, developed independently, which Arnaud Durand and Juha Kontinen use in that paper in order to characterize the expressive powers of subclasses of dependence logic in terms of the maximum allowed width of their dependence atoms.}
\begin{coro}
	Let $\phi$ be a constancy logic formula. Then $\phi$ is logically equivalent to a constancy logic formula of the form 
	\[
		\exists z_1 \ldots z_n \left( \bigwedge_{i=1}^n =\!\!(z_i) \wedge \psi(z_1 \ldots z_n)\right)
	\]
	for some tuple of variables $\tuple z = z_1 \ldots z_n$ and some first order formula $\psi$. 
\end{coro}
\begin{proof}
	Repeatedly apply Proposition \ref{propo:const_out} to ``push out'' all constancy atoms from $\phi$, thus obtaining a formula, equivalent to it, of the form
	\[
		\exists z_1 (=\!\!(z_1) \wedge \exists z_2 (=\!\!(z_2) \wedge \ldots \wedge \exists z_n(=\!\!(z_n) \wedge \psi(z_1 \ldots z_n)))
	\]
	for some first order formula $\psi(z_1 \ldots z_n)$. It is then easy to see, from the semantics of our logic, that this is equivalent to 
	\[
		\exists z_1 \ldots z_n(=\!\!(z_1) \wedge \ldots \wedge =\!\!(z_n) \wedge \psi(z_1 \ldots z_n))
	\]
	as required. 
\end{proof}
The following result shows that, over sentences, constancy logic is precisely as expressive as first order logic:
\begin{coro}
	\label{coro:const_remove}
	Let $\phi = \exists \tuple z \left(\bigwedge_i =\!\!(z_i) \wedge \psi(\tuple z)\right)$ be a constancy logic sentence in normal form.

	Then $\phi$ is logically equivalent to $\exists \tuple z \psi(\tuple z)$. 
\end{coro}
\begin{proof}
	By the rules of our semantics, $M \models_{\{\emptyset\}} \psi$ if and only if there exists a family $A_1 \ldots A_n$ of nonempty sets of elements in $\dom(M)$ such that, for 
	\[
		X = \{(z_1:= m_1 \ldots z_n:= m_n) : (m_1 \ldots m_n) \in A_1 \times \ldots \times A_n\},
	\] 
	it holds that $M \models_X \psi$. But $\psi$ is first-order, and therefore, by Proposition \ref{propo:FOflat}, this is the case if and only if for all $m_1 \in A_1, \ldots, m_n \in A_n$ it holds that $M \models_{\{(z_1:m_1, \ldots z_n:m_n)\}} \psi$.

	But then $M \models_{\{\emptyset\}} \phi$ is and only if there exist $m_1 \ldots m_n$ such that this holds;\footnote{Indeed, if this is the case we can just choose $A_1 = \{m_1\}, \ldots, A_n = \{m_n\}$, and conversely if $A_1 \ldots A_n$ exist with the required properties we can simply select arbitrary elements of them for $m_1 \ldots m_n$.} and therefore, by Proposition \ref{propo:FO_Team2Tarski}, $M \models_{\{\emptyset\}} \phi$ if and only if $M \models_\emptyset \exists z_1 \ldots z_n \psi(z_1 \ldots z_n)$ according to Tarski's semantics, or equivalently, if and only if $M \models_{\{\emptyset\}} \exists z_1 \ldots z_n \psi(z_1 \ldots z_n)$ according to team semantics. 
\end{proof}
Since, by Theorem \ref{theo:DLsent}, dependence logic is strictly stronger than first order logic over sentences, this implies that constancy logic is strictly weaker than dependence logic over sentences (and, since sentences are a particular kind of formulas, over formulas too).\\

The relation between first order logic and constancy logic, in conclusion, appears somewhat similar to that between dependence logic and independence logic - that is, in both cases we have a pair of logics which are reciprocally translatable on the level of sentences, but such that one of them is strictly weaker than the other on the level of formulas. This discrepancy between translatability on the level of sentences and translatability on the level of formulas is, in the opinion of the author, one of the most intriguing aspects of logics of imperfect information, and it deserves further investigation.
\section{Inclusion and exclusion in logic}
This section is the central part of the present work. We will begin it by recalling two forms of non-functional dependency which have been studied in Database Theory, and some of their known properties. Then we will briefly discuss their relevance in the framework of logics of imperfect information, and then, in Subsection \ref{subsect:inclog}, we will examine the properties of the logic obtained by adding atoms corresponding to the first sort of non-functional dependency to the basic language of team semantics. Afterward, in Subsection \ref{subsect:equilog} we will see that nothing is lost if we only consider a simpler variant of this kind of dependency: in either case, we obtain the same logical formalism, which - as we will see - is strictly more expressive than first order logic, strictly weaker than independence logic, but incomparable with dependence logic. In Subsection \ref{subsect:exclog}, we will then study the other notion of non-functional dependency that we are considering, and see that the corresponding logic is instead equivalent, in a very strong sense, to dependence logic; and finally, in Subsection \ref{subsect:ielogic} we will examine the logic obtained by adding \emph{both} forms of non-functional dependency to our language, and see that it is equivalent to independence logic. 
\subsection{Inclusion and exclusion dependencies}
\label{subsect:incexc}
Functional dependencies are the forms of dependency which attracted the most interest from database theorists, but they certainly are not the only ones ever considered in that field.\\ 

Therefore, studying the effect of substituting the dependence atoms with ones corresponding to other forms of dependency, and examining the relationship between the resulting logics, may be - in the author's opinion, at least -  a very promising, and hitherto not sufficiently explored, direction of research in the field of logics of imperfect information.\footnote{Apart from the present paper, \cite{engstrom10}, which introduces \emph{multivalued dependence atoms}, is also a step in this direction. The resulting ``multivalued dependence logic'' is easily seen to be equivalent to independence logic.} First of all, as previously mentioned, teams correspond to states of knowledge. But often, relations obtained from a database correspond precisely to information states of this kind;\footnote{As a somewhat naive example, let us consider the problem of finding a spy, knowing that yesterday he took a plane from London's Heathrow airport and that he had at most 100 EUR available to buy his plane ticket. We might then decide to obtain, from the airport systems, the list of the destinations of all the planes which left Heathrow yesterday and whose ticket the spy could have afforded; and this list - that is, the list of all the places that the spy might have reached - would be a state of information of the kind which we are discussing.} and therefore, some of the dependencies studied in database theory may correspond to constraints over the agent's beliefs which often occur in practice, as is certainly the case for functional dependencies.\footnote{For example, our system should be able to represent the assertion that the flight code always determines the destination of the flight.}

Moreover, and perhaps more pragmatically, database researchers have already performed a vast amount of research about the properties of many of these non-functional dependencies, and it does not seem unreasonable to hope that this might allow us to derive, with little additional effort of our own, some useful results about the corresponding logics.\\

The present paper will, for the most part, focus on \emph{inclusion} (\cite{fagin81}, \cite{casanova82}) and \emph{exclusion} (\cite{casanova83}) dependencies and on the properties of the corresponding logics of imperfect information. Let us start by recalling and briefly discussing these dependencies:

\begin{defin}[Inclusion Dependencies]
	Let $R$ be a relation, and let $\tuple x$, $\tuple y$ be tuples of attributes of $R$ of the same length. Then $R \models \tuple x \subseteq \tuple y$ if and only if $R(\tuple x) \subseteq R(\tuple y)$, where 
	\[
		R(\tuple z) = \{r(\tuple z) : r \mbox{ is a tuple in } R\}.
	\]
\end{defin}

In other words, an inclusion dependency $\tuple x \subseteq \tuple y$ states that all values taken by the attributes $\tuple x$ are also taken by the attributes $\tuple y$. It is easy to think up practical examples of inclusion dependencies: one might for instance think of the database consisting of the relations (Person, Date\_of\_Birth), (Father, Children$_F$) and (Mother, Children$_M$).\footnote{Equivalently, one may consider the Cartesian product of these relations, as per the universal relation model (\cite{fagin84}).}  Then, in order to express the statement that every father, every mother and every child in our knowledge base are people and have a date of birth, we may impose the restrictions 
\[
\left\{
\begin{array}{l}
\mbox{Father} \subseteq \mbox{Person}, ~\mbox{Mother} \subseteq \mbox{Person},\\
\mbox{Children}_F \subseteq \mbox{Person}, ~\mbox{Children}_M \subseteq \mbox{Person}
\end{array}
\right\}.
\]
Furthermore, inclusion dependencies can be used to represent the assertion that  every child has a father and a mother, or, in other words, that the attributes Children$_F$ and Children$_M$ take the same values: 
\[
\{\mbox{Children}_F \subseteq \mbox{Children}_M,~ \mbox{Children}_M \subseteq \mbox{Children}_F\}.
\]
Note, however, that inclusion dependencies do not allow us to express all ``natural'' dependencies of our example. For instance, we need to use functional dependencies in order to assert that everyone has exactly one birth date, one father and one mother:\footnote{The simplest way to verify that these conditions are not expressible in terms of inclusion dependencies is probably to observe that inclusion dependencies are \emph{closed under unions}: if the relations $R$ and $S$ respect $\tuple x \subseteq \tuple y$, so does $R \cup S$. Since functional dependencies as the above ones are clearly \emph{not} closed under unions, they cannot be represented by inclusions.}
\[
\{\mbox{Person}\rightarrow \mbox{Date\_of\_Birth},~  \mbox{Children}_F \rightarrow \mbox{Father},~ \mbox{Children}_M \rightarrow \mbox{Mother}\}.
\]

In \cite{casanova82}, a sound and complete axiom system for the implication problem of inclusion dependencies was developed. This system consists of the three following rules: 
\begin{description}
	\item[I1:] For all $\tuple x$, $\vdash \tuple x \subseteq \tuple x$;
	\item[I2:] If $|\tuple x| = |\tuple y| = n$ then, for all $m \in \mathbb N$ and all $\pi: 1 \ldots m \rightarrow 1 \ldots n$, 
		\[
			\tuple x \subseteq \tuple y \vdash x_{\pi(1)} \ldots x_{\pi(m)} \subseteq y_{\pi(1)} \ldots y_{\pi(m)}; 
		\]
	\item[I3:] For all tuples of attributes of the same length $\tuple x$, $\tuple y$, and $\tuple z$, 
		\[
			\tuple x \subseteq \tuple y, \tuple y \subseteq \tuple z \vdash \tuple x \subseteq \tuple z.
		\]
\end{description}
\begin{theo}[Soundness and completeness of inclusion axioms \cite{casanova82}]
	Let $\Gamma$ be a set of inclusion dependencies and let $\tuple x$, $\tuple y$ be tuples of relations of the same length. Then 
	\[
		\Gamma \vdash \tuple x \subseteq \tuple y
	\]
	can be derived from the axioms \textbf{I1}, \textbf{I2} and \textbf{I3} if and only if all relations which respect all dependencies of $\Gamma$ also respect $\tuple x \subseteq \tuple y$.
\end{theo}

However, the combined implication problem for inclusion and functional dependencies is undecidable (\cite{mitchell83}, \cite{chandra85}). \\

Whereas inclusion dependencies state that all values of a given tuple of attributes also occur as values of another tuple of attributes, \emph{exclusion} dependencies state that two tuples of attributes have no values in common: 
\begin{defin}[Exclusion dependencies]
		Let $R$ be a relation, and let $\tuple x$, $\tuple y$ be tuples of attributes of $R$ of the same length. Then $R \models \tuple x ~|~ \tuple y$ if and only if $R(\tuple x) \cap R(\tuple y) = \emptyset$, where 
	\[
		R(\tuple z) = \{r(\tuple z) : r \mbox{ is a tuple in } R\}.
	\]
\end{defin}

Exclusion dependencies can be thought of as a way of partitioning the elements of our domain into \emph{data types}, and of specifying which type corresponds to each attribute. For instance, in the example 
\[
\mbox{(Person, Date\_of\_birth)}\times \mbox{(Father, Children$_F$)}\times \mbox{(Mother, Children$_M$)}
\]
considered above we have two types, corresponding respectively to \emph{people} (for the attributes Person, Father, Mother, Children$_F$ and Children$_M$) and \emph{dates} (for the attribute Date\_of\_birth). The requirement that no date of birth should be accepted as a name of person, nor vice versa, can then be expressed by the set of exclusion dependencies 
\[
\{ A ~|~ \mbox{Date\_of\_birth} :  A = \mbox{Person}, \mbox{Father}, \mbox{Mother}, \mbox{Children}_M, \mbox{Children}_F\}.
\]

Other uses of exclusion dependencies are less common, but they still exist: for example, the statement that no one is both a father and a mother might be expressed as $\mbox{Father} ~|~ \mbox{Mother}$.

In \cite{casanova83}, the axiom system for inclusion dependencies was extended to deal with both inclusion and exclusion dependencies as follows:
\begin{enumerate}
	\item \emph{Axioms for inclusion dependencies:}
		\begin{description}
				\item[I1:] For all $\tuple x$, $\vdash \tuple x \subseteq \tuple x$;
				\item[I2:] If $|\tuple x| = |\tuple y| = n$ then, for all $m \in \mathbb N$ and all $\pi: 1 \ldots m \rightarrow 1 \ldots n$, 
		\[
			\tuple x \subseteq \tuple y \vdash x_{\pi(1)} \ldots x_{\pi(m)} \subseteq y_{\pi(1)} \ldots y_{\pi(m)}; 
		\]
				\item[I3:] For all tuples of attributes of the same length $\tuple x$, $\tuple y$ and $\tuple z$, 
		\[
			\tuple x \subseteq \tuple y, \tuple y \subseteq \tuple z \vdash \tuple x \subseteq \tuple z;
		\]
		\end{description}
	\item \emph{Axioms for exclusion dependencies:}
		\begin{description}
			\item[E1:] For all $\tuple x$ and $\tuple y$ of the same length, $\tuple x ~|~ \tuple y \vdash \tuple y ~|~ \tuple x$;
			\item[E2:] If $|\tuple x| = |\tuple y| = n$ then, for all $m \in \mathbb N$ and all $\pi: 1 \ldots m \rightarrow 1 \ldots n$, 
		\[
			x_{\pi(1)} \ldots x_{\pi(m)} ~|~ y_{\pi(1)} \ldots y_{\pi(m)} \vdash \tuple x ~|~ \tuple y;
		\]
			\item[E3:] For all $\tuple x$, $\tuple y$ and $\tuple z$ such that $|\tuple y| = |\tuple z|$, $\tuple x ~|~ \tuple x \vdash \tuple y ~|~ \tuple z$;
		\end{description}
	\item \emph{Axioms for inclusion/exclusion interaction:}
		\begin{description}
			\item[IE1:] For all $\tuple x$, $\tuple y$ and $\tuple z$ such that $|\tuple y| = |\tuple z|$, $\tuple x ~|~ \tuple x \vdash \tuple y \subseteq \tuple z$;
			\item[IE2:] For all $\tuple x, \tuple y, \tuple z, \tuple w$ of the same length, $\tuple x ~|~ \tuple y, \tuple z \subseteq \tuple x, \tuple w \subseteq \tuple y \vdash \tuple z ~|~ \tuple w$.
		\end{description}
\end{enumerate}
\begin{theo}[\cite{casanova83}]
	The above system is sound and complete for the implication problem for inclusion and exclusion dependencies. 
\end{theo}

It is not difficult to transfer the definitions of inclusion and exclusion dependencies to team semantics, thus obtaining \emph{inclusion atoms} and \emph{exclusion atoms}:
\begin{defin}[Inclusion and exclusion atoms]
	\label{def:inc_exc}
	Let $M$ be a first order model, let $\tuple t_1$ and $\tuple t_2$ be two finite tuples of terms of the same length over the signature of $M$, and let $X$ be a team whose domain contains all variables occurring in $\tuple t_1$ and $\tuple t_2$. Then 
	\begin{description}
	\item[TS-inc:] $M \models_X \tuple t_1 \subseteq \tuple t_2$ if and only if for every $s \in X$ there exists a $s' \in X$ such that $\tuple t_1\langle s\rangle = \tuple t_2\langle s'\rangle$;
	\item[TS-exc:] $M \models_X \tuple t_1 ~|~ \tuple t_2$ if and only if for all $s, s' \in X$, $\tuple t_1\langle s\rangle \not= \tuple t_2\langle s'\rangle$.
	\end{description}
\end{defin}

Returning for a moment to the agent metaphor, the interpretation of these conditions is as follows.\\

A team $X$ satisfies $\tuple t_1 \subseteq \tuple t_2$ if and only if all possible values that the agent believes possible for $\tuple t_1$ are also believed by him or her as possible for $\tuple t_2$ - or, by contraposition, that the agent cannot exclude any value for $\tuple t_2$ which he cannot also exclude as a possible value for $\tuple t_1$. In other words, from this point of view an inclusion atom is a way of specify a state of \emph{ignorance} of the agent: for example, if the agent is a chess player who is participating to a tournament, we may want to represent the assertion that the agent \emph{does not know} whether he will play against a given opponent using the black pieces or the white ones. In other words, if he believes that he \emph{might} play against a given opponent when using the white pieces, he should also consider it possible that he played against him or her using the black ones, and vice versa; or, in our formalism, that his belief set satisfies the conditions
\begin{align*}
&\mbox{Opponent\_as\_White} \subseteq \mbox{Opponent\_as\_Black},\\
&\mbox{Opponent\_as\_Black} \subseteq \mbox{Opponent\_as\_White}.
\end{align*}

This very example can be used to introduce a new dependency atom $\tuple t_1 \bowtie \tuple t_2$, which might perhaps be called an \emph{equiextension atom}, with the following rule:
\begin{defin}[Equiextension atoms]
	\label{def:equi}
	Let $M$ be a first order model, let $\tuple t_1$ and $\tuple t_2$ be two finite tuples of terms of the same length over the signature of $M$, and let $X$ be a team whose domain contains all variables occurring in $\tuple t_1$ and $\tuple t_2$. Then 
	\begin{description}
	\item[TS-equ:] $M \models_X \tuple t_1 \bowtie \tuple t_2$ if and only if $X(\tuple t_1) = X(\tuple t_2)$.
	\end{description}
\end{defin}
It is easy to see that this atom is different, and strictly weaker, from the first order formula 
\[
	\tuple t_1 = \tuple t_2 := \bigwedge_i ((\tuple t_1)_i = (\tuple t_2)_i).
\]
Indeed, the former only requires that the sets of all possible values for $\tuple t_1$ and for $\tuple t_2$ are the same, while the latter requires that $\tuple t_1$ and $\tuple t_2$ coincide in all possible states of things: and hence, for example, the team $X = \{(x:0,y:1),(x:1, y:0)\}$ satisfies $x \bowtie y$ but not $x = y$. 

As we will see later, it is possible to recover inclusion atoms from equiextension atoms and the connectives of our logics.\\

On the other hand, an exclusion atom specifies a state of \emph{knowledge}. More precisely, a team $X$ satisfies $\tuple t_1 ~|~ \tuple t_2$ if and only if the agent can confidently exclude all values that he believes possible for $\tuple t_1$ from the list of the possible values for $\tuple t_2$. For example, let us suppose that our agent is also aware that a boxing match will be had at the same time of the chess tournament, and that he knows that no one of the participants to the match will have the time to play in the tournament too - he has seen the lists of the participants to the two events, and they are disjoint. Then, in particular, our agent knows that no potential winner of the boxing match is also a potential winner of the chess tournament, even know he is not aware of who these winners will be. In our framework, this can be represented by stating our agent's beliefs respect the exclusion atom 
\[
	\mbox{Winner\_Boxing} ~|~ \mbox{Winner\_Chess}.
\]

This is a different, and stronger, condition than the first order expression $\mbox{Winner\_Boxing} \not = \mbox{Winner\_Chess}$: indeed, the latter merely requires that, in any possible state of things, the winners of the boxing match and of the chess tournament are different, while the former requires that \emph{no possible} winner of the boxing match is a potential winner for the chess tournament. So, for example, only the first condition excludes the scenario in which  our agent does not know whether T. Dovramadjiev, a Bulgarian chessboxing champion, will play in the chess tournament or in the boxing match, represented by the team of the form
\[
	X = \begin{array}{c | l l l}
		& \mbox{Winner\_Boxing} & \mbox{Winner\_Chess} & \ldots\\
		\hline
		s_0 & \mbox{T. Dovramadjiev} & \mbox{V. Anand} & \ldots\\
		s_1 & \mbox{T. Woolgar} & \mbox{T. Dovramadijev} & \ldots\\
		\ldots & \ldots & \ldots
	\end{array}
\]
\subsection{Inclusion logic}
\label{subsect:inclog}
In this section, we will begin to examine the properties of \emph{inclusion logic} - that is, the logic obtained adding to (team) first order logic the \emph{inclusion atoms} $\tuple t_1 \subseteq \tuple t_2$ with the semantics of Definition \ref{def:inc_exc}.\\

A first, easy observation is that this logic does not respect the downwards closure property. For example, consider the two assignments $s_0 = (x:0, y:1)$ and $s_1 = (x:1,y:0)$: then, for $X = \{s_0, s_1\}$ and $Y = \{s_0\}$, it is easy to see by rule \textbf{TS-inc} that $M \models_X x \subseteq y$ but $M \not \models_Y x \subseteq y$.\\

Hence, the proof of Proposition \ref{propo:laxeqstrict} cannot be adapted to the case of inclusion logic. The question then arises whether inclusion logic with strict semantics and inclusion logic with lax semantics are different; and, as the next two propositions will show, this is indeed the case. 
\begin{propo}
	\label{propo:dislaxneqstr}
	There exist a model $M$, a team $X$ and two formulas $\psi$ and $\theta$ of inclusion logic such that $M \models_X^L \psi \vee \theta$ but $M \not \models_X^S \psi \vee \theta$.
\end{propo}
\begin{proof}
	Let $\dom(M) = \{0,1, 2, 3, 4\}$, let $X$ be the team
	\[
		X = \begin{array}{c | c c c}
				~ & x & y & z\\
				\hline
				s_0 & 0 & 1 & 2\\
				s_1 & 1 & 0 & 3\\
				s_2 & 4 & 3 & 0
			\end{array}
	\]
	and let $\psi = x \subseteq y$, $\theta = y \subseteq z$. 
	\begin{itemize}
		\item $M \models_X^L \psi \vee \theta$:\\
			Let $Y = \{s_0, s_1\}$ and $Z = \{s_1, s_2\}$. Then $Y \cup Z = X$, $Y(x) = \{0,1\} = Y(y)$ and $Z(y) = \{0,3\} = Z(z)$.\\

			Hence, $M \models_Y^L x \subseteq y$ and $M \models_Z^L y \subseteq z$, and therefore $M \models_X^L x \subseteq y \vee y \subseteq z$ as required.
		\item $M \not \models_X^S \psi \vee \theta$:\\
			Suppose that $X = Y \cup Z$, $Y \cap Z = \emptyset$, $M \models_X^S x \subseteq y$ and $M \models_Z^S y \subseteq z$.\\

			Now, $s_2$ cannot belong in $Y$, since $s_2(x) = 4$ and $s_i(y) \not = 4$ for all assignments $s_i$; therefore, we necessarily have that $s_2 \in Z$. But since $M \models_Z^S y \subseteq z$, this implies that there exists an assignment $s_i \in Y$ such that $s_i(z) = s_2(y) = 3$. The only such assignment in $X$ is $s_1$, and therefore $s_1 \in Y$. \\

			Analogously, $s_0$  cannot belong in $Z$: indeed, $s_0(y) = 1 \not = s_i(z)$ for all $i \in 0 \ldots 2$. Therefore, $s_0 \in Y$; and since $M \models_Y^S x \subseteq y$, there exists an $s_i \in Y$ with $s_i(y) = s_0(x) = 0$. But the only such assignment in $X$ is $s_1$, and therefore $s_1 \in Y$. \\

			In conclusion, $Y = \{s_0, s_1\}$, $Z = \{s_1, s_2\}$ and $Y \cap Z = \{s_1\} \not = \emptyset$, which contradicts our hypothesis. 
	\end{itemize}
\end{proof}
\begin{propo}
	\label{propo:exlaxneqstr}
	There exist a model $M$, a team $X$ and a formula $\phi$ of inclusion logic such that $M \models_X^L \exists x \phi$ 
	but $M \not \models_X^S \exists x \phi$.
\end{propo}
\begin{proof}
	Let $\dom(M) = \{0,1\}$, let $X$ be the team
	\[
		X = \begin{array}{c | c  c }
				~ & y & z\\
				\hline
				s_0 & 0 & 1
			\end{array}
	\]
	and let $\phi$ be $y \subseteq x \wedge z \subseteq x$. 
	\begin{itemize}
		\item $M \models_X^L \exists x \phi$: \\
			Let $H: X \rightarrow \part(\dom(M))$ be such that $H(s_0) = \{0,1\}$.

Then
			\[
				X[H/x] = \begin{array}{c | c c c}
				~ & y & z & x\\
				\hline
				s'_0 & 0 & 1 & 0\\
				s'_1 & 0 & 1 & 1\\
				\end{array}
			\]
			and hence $X[H/x](y), X[H/x](z) \subseteq X[H/x](x)$, as required.
		\item $M \not \models_X^S \exists x \psi$:\\
			Let $F$ be any function from $X$ to $\dom(M)$. Then 
				\[
				X[F/x] = \begin{array}{c | c c c}
				~ & y & z & x\\
				\hline
				s''_0 & 0 & 1 & F(s_0)\\
				\end{array}
				\]
				But $F(s_0) \not = 0$ or $F(s_0) \not = 1$; and in the first case $M \not \models_{X[F/x]}^S y \subseteq x$, while in the second one $M \not \models_{X[F/x]}^S z \subseteq x$.
	\end{itemize}
\end{proof}

Therefore, when studying the properties inclusion logic it is necessary to specify whether we are are using the strict or the lax semantics for disjunction and existential quantification. However, only one of these choices preserves \emph{locality} in the sense of Theorem \ref{theo:DLloc}, as the two following results show: 
\begin{propo}
	\label{propo:strict_nonlocal}
	The strict semantics does not respect locality in inclusion logic (or in any extension thereof). In other words, there exists a a model $M$, a team $X$ and two formulas $\psi$ and $\theta$ such that $M \models_X^S \psi \vee \theta$, but for $X' = X_{\upharpoonright \free(\phi \vee \psi)}$ it holds that  $M \not \models_{X'}^S \psi \vee \theta$ instead; and analogously, there exists a model $M$, a team $X$ and a formula $\xi$ such that $M \models_X^S \exists x \xi$, but for $X' = X_{\upharpoonright \free(\exists x \xi)}$  we have that $M \not \models_{X'}^S \exists \xi$ instead. 
\end{propo}
\begin{proof}
	\begin{enumerate}
		\item Let $\dom(M) = \{0 \ldots 4\}$, let $\psi$ and $\theta$ be $x \subseteq y$ and $y \subseteq z$ respectively, and let 
			\[
						X = \begin{array}{c | c c c c}
				~ & x & y & z & u\\
				\hline
				s_0 & 0 & 1 & 2 & 0\\
				s_1 & 1 & 0 & 3 & 0\\
				s_2 & 1 & 0 & 3 & 1\\
				s_3 & 4 & 3 & 0 & 0
			\end{array}
			\]
			Then $M \models_X^S \psi \vee \theta$: indeed, for $Y = \{s_0, s_1\}$ and $Z = \{s_2, s_3\}$ we have that $X = Y \cup Z$, $Y \cap Z = \emptyset$, $M \models_Y \psi$ and $M \models_Z \theta$, as required. However, the restriction $X'$ of $X$ to $\free(\psi \vee \theta) = \{x,y,z\}$ is the team considered in the proof of Proposition \ref{propo:dislaxneqstr}, and - as was shown in that proof - $M \not \models_X^S \psi \vee \theta$. 
		\item Let $\dom(M) = \{0,1\}$,  let $\xi$ be $y \subseteq x \wedge z \subseteq x$, and let 
			\[
				X = \begin{array}{c | c c c}
					~ & y & z & u\\
					\hline
					s_0 & 0 & 1 & 0\\
					s_1 & 0 & 1 & 1
				\end{array}
			\]
			Then $M \models_X^S \exists x \xi$: indeed, for $F: X \rightarrow \dom(M)$ defined as 
			\begin{align*}
				&F(s_0) = 0;\\
				&F(s_1) = 1;\\
			\end{align*}
			we have that 
			\[
				X[F/x] = \begin{array}{c | c c c c}
					~ & y & z & u & x\\
					\hline
					s'_0 & 0 & 1 & 0 & 0\\
					s'_1 & 0 & 1 & 1 & 1
				\end{array}
			\]
			and it is easy to check that this team satisfies $\xi$. However, the restriction $X'$ of $X$ to $\free(\exists x \xi) = \{y, z\}$ is the team considered in the proof of Proposition \ref{propo:exlaxneqstr}, and - again, as shown in that proof - $M \not \models_X^S \exists x \psi$.
	\end{enumerate}
\end{proof}
\begin{theo}[Inclusion logic with lax semantics is local]
	\label{theo:strict_local}
	Let $M$ be a first order model, let $\phi$ be any inclusion logic formula, and let $V$ be a set of variables with $\free(\phi) \subseteq V$. Then, for all suitable teams $X$,
	\[
	M \models_X^L \phi \Leftrightarrow M \models_{X_{\upharpoonright V}}^L \phi
	\]
\end{theo}
\begin{proof}
	The proof is by structural induction on $\phi$. 

	In Section \ref{subsect:ielogic}, Theorem \ref{theo:ielocal}, we will prove the same result for an extension of inclusion logic; so we refer to that theorem for the details of the proof.
\end{proof}

Since, as we saw, inclusion logic is not downwards closed, by Theorem \ref{theo:DLdc} it is not contained in dependence logic. It is then natural to ask whether dependence logic is contained in inclusion logic, or if dependence and inclusion logic are two incomparable extensions of first order logic. \\

This is answered by the following result, and by its corollary: 
\begin{theo}[Union closure for inclusion logic]
	Let $\phi$ be any inclusion logic formula, let $M$ be a first order model and let $(X_i)_{i \in I}$ be a family of teams with the same domain such that $M \models_{X_i} \phi$ for all $i \in I$. Then, for $X = \bigcup_{i \in I} X_i$, we have that $M \models_X \phi.$
\end{theo}
\begin{proof}
	By structural induction on $\phi$.
	\begin{enumerate}
		\item If $\phi$ is a first order literal, this is obvious.
		\item Suppose that $M \models_{X_i} \tuple t_1 \subseteq \tuple t_2$ for all $i \in I$. Then $M \models_X \tuple t_1 \subseteq \tuple t_2$. Indeed, let $s \in X$: then $s \in X_i$ for some $i \in I$, and hence there exists another $s' \in X_i$ with $s'(\tuple t_2) = s(\tuple t_1)$. Since $X_i \subseteq X$ we then have that $s' \in X$, as required. 
		\item Suppose that $M \models_{X_i} \psi \vee \theta$ for all $i \in I$. Then each $X_i$ can be split into two subteams $Y_i$ and $Z_i$ with $M \models_{Y_i} \psi$ and $M \models_{Z_i} \theta$. Now, let $Y = \bigcup_{i \in I} Y_i$ and $Z = \bigcup_{i \in I} Z_i$: by induction hypothesis, $M \models_Y \psi$ and $M \models_Z \theta$. Furthermore, $Y \cup Z = \bigcup_{i \in I} Y_i ~\cup~ \bigcup_{i \in I} Z_i = \bigcup_{i \in I} (Y_i \cup Z_i) = X$, and hence $M \models_X \psi \vee \theta$, as required.
		\item Suppose that $M \models_{X_i} \psi \wedge \theta$ for all $i \in I$. Then for all such $i$, $M \models_{X_i} \psi$ and $M \models_{X_i} \theta$; but then, by induction hypothesis, $M \models_X \psi$ and $M \models_X \theta$, and therefore $M \models_X \psi \wedge \theta$. 
		\item Suppose that $M \models_{X_i} \exists x \psi$ for all $i \in I$, that is, that for all such $i$ there exists a function $H_i: X_i \rightarrow \part(\dom(M)) \backslash \{\emptyset\}$ such that $M \models_{X_i[H_i/x]} \psi$. Then define the function $H : X \rightarrow \part(\dom(M)) \backslash \{\emptyset\}$ so that, for all $s \in X$, $H(s) = \bigcup \{H_i(s): s \in X_i\}$. Now, $X[H/x] = \bigcup_{i \in I} (X_i[H_i/x])$, and hence by induction hypothesis $M \models_{X[H/x]} \psi$, and therefore $M \models_X \exists x \psi$. 
		\item Suppose that $M \models_{X_i} \forall x \psi$ for all $i \in I$, that is, that $M \models_{X_i[M/x]} \psi$ for all such $i$. Then, since $\bigcup_{i \in I} (X_i[M/x]) = \left(\bigcup_{i \in I} X_i\right)[M/x] = X[M/x]$, by induction hypothesis $M \models_{X[M/x]} \psi$ and therefore $M \models_X \forall x \psi$, as required.
	\end{enumerate}
\end{proof}
\begin{coro}
	There exist constancy logic formulas which are not equivalent to any inclusion logic formula.
\end{coro}
\begin{proof}
	This follows at once from the fact that the constancy atom $=\!\!(x)$ is not closed under unions.\\

	Indeed, let $M$ be any model with two elements $0$ and $1$ in its domain, and consider the two teams $X_0 = \{(x:0)\}$ and $X_1 = \{(x:1)\}$: then $M \models_{X_0} =\!\!(x)$ and $M \models_{X_1} =\!\!(x)$, but $M \not \models_{X_0 \cup X_1} =\!\!(x)$. 
\end{proof}
Therefore, not only inclusion logic does not contain dependence logic, it does not even contain constancy logic!\\

Now, by Theorem \ref{theo:DL2IL} we know that dependence logic is properly contained in independence logic. As the following result shows, inclusion logic is also (properly, because dependence atoms are expressible in independence logic) contained in independence logic: 
\begin{theo}
	\label{theo:Inc2IL}
	Inclusion atoms are expressible in terms of independence logic formulas. More precisely, an inclusion atom $\tuple t_1 \subseteq \tuple t_2$ is equivalent to the independence logic formula
\[
	\phi := \forall v_1 v_2 \tuple z ((\tuple z \not = \tuple t_1 \wedge \tuple z \not = \tuple t_2) \vee (v_1 \not = v_2 \wedge \tuple z \not = \tuple t_2) \vee ((v_1 = v_2 \vee \tuple z = \tuple t_2) \wedge \indepc{\tuple z}{v_1 v_2})).
\]
where $v_1$, $v_2$ and $\tuple z$ do not occur in $\tuple t_1$ or $\tuple t_2$ and where, as in \cite{gradel10}, $\indepc{\tuple z}{v_1 v_2}$ is a shorthand for $\indep{\emptyset}{\tuple z}{v_1v_2}$.
\end{theo}
\begin{proof}
Suppose that $M \models_X \tuple t_1 \subseteq \tuple t_2$. Then split the team $X' = X[M/v_1 v_2 \tuple z]$ into three teams $Y$, $Z$ and $W$ as follows: 
\begin{itemize}
\item $Y = \{ s \in X' : s(\tuple z) \not = \tuple t_1\langle s\rangle \mbox{ and } s(\tuple z) \not = \tuple t_2\langle s\rangle\}$;
\item $Z = \{ s \in X' : s(v_1) \not = s(v_2) \mbox{ and } s(\tuple z) \not = \tuple t_2\langle s\rangle\}$;
\item $W = X' \backslash (Y \cup Z) = \{s \in X' : s(\tuple z) = \tuple t_2\langle s\rangle \mbox{ or } (s(\tuple z) = \tuple t_1\langle s\rangle \mbox{ and } s(v_1) = s(v_2))\}$. 
\end{itemize}
Clearly, $X' = Y \cup Z \cup W$, $M \models_Y z \not = t_1 \wedge z \not = t_2$ and $M \models_Z v_1 \not = v_2 \wedge z \not = t_2$; hence, if we can prove that 
\[
	M \models_W ((v_1 = v_2 \vee \tuple z = \tuple t_2)) \wedge \indepc{\tuple z}{v_1 v_2}
\]
we can conclude that $M \models_X \phi$, as required. 

Now, suppose that $s \in W$ and $s(v_1) \not = s(v_2)$: then necessarily $s(\tuple z) = \tuple t_2$, since otherwise we would have that $s \in Z$ instead. Hence, the first conjunct $v_1 = v_2 \vee \tuple z = \tuple t_2$ is satisfied by $W$.

Now, consider two assignments $s$ and $s'$ in $W$: in order to conclude this direction of the proof, we need to show that there exists a $s'' \in W$ such that $s''(\tuple z) = s(\tuple z)$ and $s''(v_1 v_2) = s'(v_1 v_2)$. There are two distinct cases to examine: 
\begin{enumerate}
	\item If $s(\tuple z) = \tuple t_2 \langle s \rangle$, consider the assignment 
		\[
			s'' = s[s'(v_1)/v_1][s'(v_2)/v_2]:
		\]
		by construction, $s'' \in X'$. Furthermore, since $s''(\tuple z) = \tuple t_2\langle s\rangle = \tuple t_2\langle s''\rangle$, $s''$ is neither in $Y$ nor in $Z$. Hence, it is in $W$, as required. 
	\item If $s(\tuple z) \not = \tuple t_2\langle s\rangle$ and $s \in W$, then necessarily $s(\tuple z) = \tuple t_1\langle s \rangle$ and $s(v_1) = s(v_2)$. 

	Since $s \in W \subseteq X[M/v_1 v_2 \tuple z]$, there exists an assignment $o \in X$ such that 
	\[
		\tuple t_1\langle o\rangle = \tuple t_1 \langle s\rangle = s(\tuple z);
	\]
	and since $M \models_X \tuple t_1 \subseteq \tuple t_2$, there also exist an assignment $o' \in X$ such that 
	\[
		\tuple t_2\langle o'\rangle = \tuple t_1\langle o\rangle = s(\tuple z).
	\]

	Now consider the assignment $s'' = o'[s'(v_1)/v_1][s'(v_2)/v_2][s(\tuple z)/\tuple z]$: by construction, $s'' \in X'$, and since 
	\[
		s''(\tuple z) = s(\tuple z) = \tuple t_2 \langle o'\rangle = \tuple t_2\langle s''\rangle
	\]
	we have that $s'' \in W$, that $s''(\tuple z) = s(\tuple z)$ and that $s''(v_1 v_2) = s'(v_1 v_2)$, as required. 
\end{enumerate}

Conversely, suppose that $M \models_X \phi$, let $0$ and $1$ be two distinct elements of the domain of $M$, and let $s \in X$.

By the definition of $\phi$, the fact that $M \models_X \phi$ implies that the team $X[M/v_1v_2\tuple z]$ can be split into three teams $Y$, $Z$ and $W$ such that 
\begin{align*}
	& M \models_Y \tuple z \not = \tuple t_1 \wedge \tuple z \not = \tuple t_2;\\
	& M \models_Z v_1 \not = v_2 \wedge \tuple z \not = \tuple t_2;\\
	& M \models_W (v_1 = v_2 \vee \tuple z = \tuple t_2) \wedge \indepc{\tuple z}{v_1 v_2}.
\end{align*}
 Then consider the assignments 
\[
	h = s[0/v_1][0/v_2][\tuple t_1\langle s\rangle/\tuple z]
\]
and 
\[
	h' = s[0/v_1][1/v_2][\tuple t_2\langle s\rangle/\tuple z]
\]
Clearly, $h$ and $h'$ are in $X[M/v_1v_2\tuple z]$. However, neither of them is in $Y$, since $h(\tuple z) = \tuple t_1\langle h\rangle$ and $h'(\tuple z) = \tuple t_2\langle h'\rangle$, nor in $Z$, since $h(v_1) = h(v_2)$ and, again, since $h'(\tuple z) = \tuple t_2\langle h'\rangle$. Hence, both of them are in $W$. 

But we know that $M \models_W \indepc{\tuple z}{v_1 v_2}$, and thus there exists an assignment $h'' \in W$ with 
\[
	h''(\tuple z) = h(\tuple z) = \tuple t_1 \langle s \rangle
\]
and 
\[
	h''(v_1 v_2) = h'(v_1 v_2) = 01.
\]

Now, since $h''(v_1) \not = h''(v_2)$, since $h'' \in W$ and since
\[
	M \models_W v_1 = v_2 \vee \tuple z = \tuple t_2,	
\]
it must be the case that $h''(\tuple z) = \tuple t_2\langle h''\rangle$. 

Finally, this $h''$ corresponds to some $s'' \in X$; and for this $s''$,
\[
	\tuple t_2\langle s''\rangle = \tuple t_2\langle h''\rangle =  h''(\tuple z) = h(\tuple z) = \tuple t_1\langle s\rangle.
\]
This concludes the proof.
\end{proof}

The relations between first order (team) logic, constancy logic, dependence logic, inclusion logic and independence logic discovered so far are then represented by Figure \ref{fig:logrel1}.
\begin{figure}
\begin{center}
	\epsfig{file=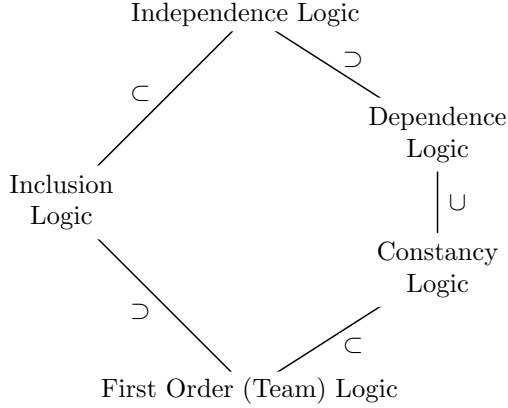}
\end{center}
\caption{Translatability relations between logics (wrt formulas)}
	\label{fig:logrel1}
\end{figure}$\\$

However, things change if we take in consideration the the expressive power of these logics with respect to their sentences only. Then, as we saw, first order logic and constancy logic have the same expressive power, in the sense that every constancy logic formula is equivalent to some first order formula and vice versa, and so do dependence and independence logic. What about inclusion logic sentences? 

At the moment, relatively little is known by the author about this. In essence, all that we know is the following result: 
\begin{propo}
Let $\psi(\tuple x, \tuple y)$ be any first order formula, where $\tuple x$ and $\tuple y$ are tuples of disjoint variables of the same arity. Furthermore, let $\psi'(\tuple x, \tuple y)$ be the result of writing $\lnot \psi(\tuple x, \tuple y)$ in negation normal form. Then, for all suitable models $M$ and all suitable pairs $\tuple a$, $\tuple b$ of constant terms of the model, 
\[
	M \models_{\{\emptyset\}} \exists \tuple z (\tuple a \subseteq \tuple z  \wedge \tuple z \not = \tuple b \wedge \forall \tuple w(\psi'(\tuple z, \tuple w) \vee \tuple w \subseteq \tuple z))
\]
if and only if $M \models \lnot [\traclo_{\tuple x, \tuple y} ~\psi](\tuple a, \tuple b)$, that is, if and only if the pair of tuples of elements corresponding to $(\tuple a, \tuple b)$ is not in the transitive closure of $\{(\tuple m_1, \tuple m_2) : M \models \psi(\tuple m_1, \tuple m_2)\}$.
\end{propo}
\begin{proof}
Suppose that $M \models_{\{\emptyset\}} \exists \tuple z (\tuple a \subseteq \tuple z  \wedge \tuple z \not = \tuple b \wedge \forall \tuple w(\psi'(\tuple z, \tuple w) \vee \tuple w \subseteq \tuple z))$. Then, by definition, there exists a tuple of functions $\tuple H = H_1 \ldots H_n$ such that 
\begin{enumerate}
	\item $M \models_{\{\emptyset\}[\tuple H/\tuple z]} \tuple a \subseteq \tuple z$, that is, $\tuple a \in \tuple H(\{\emptyset\})$;
	\item $M \models_{\{\emptyset\}[\tuple H/\tuple z]} \tuple z \not = \tuple b$, and therefore $\tuple b \not \in \tuple H(\{\emptyset\})$;
	\item $M \models_{\{\emptyset\}[\tuple H/\tuple z][\tuple M/\tuple w]} \psi'(\tuple z, \tuple w) \vee \tuple w \subseteq \tuple z$.
\end{enumerate}
Now, the third condition implies that whenever $M \models \psi(\tuple m_1, \tuple m_2)$ and $\tuple m_1$ is in $\tuple H(\{\emptyset\})$, $\tuple m_2$ is in $\tuple H(\{\emptyset\})$ too. Indeed, let $Y = \{\emptyset\}[\tuple H/\tuple z][\tuple M/\tuple w]$: then, by the semantics of our logic, we know that $Y = Y_1 \cup Y_2$ for two subteams $Y_1$ and $Y_2$ such that $M \models_{Y_1} \psi'(\tuple z, \tuple w)$ and $M \models_{Y_2} \tuple w \subseteq \tuple z$. But $\psi'$ is logically equivalent to the negation of $\psi$, and therefore we know that, for all $s \in Y_1$, $M \not \models \psi(s(\tuple z), s(\tuple w))$ in the usual Tarskian semantics.\\ 

Suppose now that $\tuple m_1 \in \tuple H(\{\emptyset\})$ and that $M \models \psi(\tuple m_1, \tuple m_2)$. Then $s = (\tuple z:=\tuple m_1, \tuple w:=\tuple m_2)$ is in $Y$; but it cannot be in $Y_1$, as we saw, and hence it must belong to $Y_2$. But $M \models_{Y_2} \tuple w \subseteq \tuple z$, and therefore there exists another assignment $s' \in Y_2$ such that $s'(\tuple z) = s(\tuple w) = \tuple m_2$. But we necessarily have that $s'(\tuple z) \in \tuple H(\{\emptyset\})$, and therefore $\tuple m_2 \in \tuple H(\{\emptyset\})$, as required.

So, $\tuple H(\{\emptyset\})$ is an set of tuples of elements of our models which contains the interpretation of $\tuple a$ but not that of $\tuple b$ and such that 
\[
	\tuple m_1 \in H(\{\emptyset\}), M \models \psi(\tuple m_1), \tuple M_2 \Rightarrow \tuple m_2 \in H(\{\emptyset\}).
\] 
This implies that $M \models \lnot [\traclo_{\tuple x, \tuple y} ~\psi](\tuple a, \tuple b)$, as required.\\

Conversely, suppose that $M \models \lnot [\traclo_{\tuple x, \tuple y} ~\psi](\tuple a, \tuple b)$: then there exists a set $A$ of tuples of elements of the domain of $M$ which contains the interpretation of $\tuple a$ but not that of $\tuple b$, and such that it is closed by transitive closure for $\psi(\tuple x, \tuple y)$. Then, by choosing the functions $\tuple H$ so that $\tuple h(\{\emptyset\}) = A$, it is easy to verify that $M$ satisfies our inclusion logic sentence.
\end{proof}
As a corollary, we have that inclusion logic is strictly more expressive than first order logic over sentences: for example, for all finite linear orders $M = (\dom(M), <, S, 0, e)$, where $S$ is the successor function, $0$ is the first element of the linear order and $e$ is the last one, we have that 
\[
	M \models \exists z (0 \subseteq z \wedge z \not = e \wedge \forall w ( w \not = S(S(z)) \vee w \subseteq z))
\]
if and only if $|M|$ is odd. It is not difficult to see, for example through the Ehrenfeucht-Fra\"iss\'e method (\cite{hodges97b}), that this property is not expressible in first order logic. 
\subsection{Equiextension logic}
\label{subsect:equilog}
Let us now consider \emph{equiextension logic}, that is, the logic obtained by adding to first order logic (with the lax team semantics) equiextension atoms $\tuple t_1 \bowtie \tuple t_2$ with the semantics of Definition \ref{def:equi}.

It is easy to see that equiextension logic is contained in inclusion logic: 
\begin{propo}
	Let $\tuple t_1$ and $\tuple t_2$ be any two tuples of terms of the same length. Then, for all suitable models $M$ and teams $X$, 
	\[
		M \models_X \tuple t_1 \bowtie \tuple t_2 \Leftrightarrow M \models_X \tuple t_1 \subseteq \tuple t_2 \wedge \tuple t_2 \subseteq \tuple t_1. 
	\]
\end{propo}
\begin{proof}
	Obvious.
\end{proof}

Translating in the other direction, however, requires a little more care:
\begin{propo}
	Let $\tuple t_1$ and $\tuple t_2$ be any two tuples of terms of the same length. Then, for all suitable models $M$ and teams $X$, $M \models_X \tuple t_1 \subseteq \tuple t_2$ if and only if 
	\[
		M \models_X \forall u_1 u_2 \exists \tuple z (\tuple t_2 \bowtie \tuple z \wedge (u_1 \not = u_2 \vee \tuple z = \tuple t_1))
	\]
	where $u_1, u_2$ and $\tuple z$ do not occur in $\tuple t_1$ and $\tuple t_2$.
\end{propo}
\begin{proof}
	Suppose that $M \models_X \tuple t_1 \subseteq \tuple t_2$. Then let $X' = X[M/u_1 u_2]$, and pick the tuple of functions $\tuple H$ used to choose $\tuple z$ so that
	\[\tuple H(s) = \left\{\begin{array}{l l}
		\{\tuple t_1 \langle s\rangle\}, & \mbox{ if } s(\tuple u_1) = s(\tuple u_2);\\
	 \{\tuple t_2\langle s\rangle\}, & \mbox{ otherwise}
 \end{array}\right.
	 \]
	 for all $s \in X'$.\footnote{As an aside, it can be observed that, since $\tuple H$ always selects singletons, this whole argument can be adapted to the case of strict semantics without any difficulties. Therefore, strict equiextension logic is equivalent to strict inclusion logic and, by Proposition \ref{propo:strict_nonlocal}, does not satisfy locality either.}\\

	Then, for $Y = X'[\tuple H/\tuple z]$, by definition we have that $M \models_{Y} u_1 \not = u_2 \vee \tuple z = \tuple t_1$, and it only remains to verify that $M \models_{Y} \tuple t_2 \bowtie \tuple z$, that is, that $Y(\tuple t_2) = Y(\tuple z)$.
	\begin{itemize}
		\item $Y(\tuple t_2) \subseteq Y(\tuple z)$:\\
			Let $h \in Y$. Then there exists an assignment $s \in X$ with $\tuple t_2\langle s\rangle = \tuple t_2\langle h\rangle$. Now let $0$ and $1$ be two distinct elements of $M$, and consider  the assignment $h' = s[0/u_1][1/u_2][\tuple H/\tuple z]$. By construction, $h' \in Y$; and furthermore, by the definition of $\tuple H$ we have that $h'(\tuple z) = \tuple t_2\langle s \rangle = \tuple t_2\langle h\rangle$, as required.
		\item $Y(\tuple z) \subseteq Y(\tuple t_2)$:\\
			Let $h \in Y$. Then, by construction, $h(\tuple z)$ is $\tuple t_1\langle h\rangle$ or $\tuple t_2 \langle h\rangle$. But since $X(\tuple t_1) \subseteq X(\tuple t_2)$, in either case there exists an assignment $s \in X$ such $\tuple t_2\langle s\rangle = h(\tuple z)$. Now consider $h' = s[0/u_1][1/u_2][\tuple H/\tuple z]$: again, $h' \in Y$ and $h'(\tuple z) = \tuple t_2\langle h'\rangle = \tuple t_2\langle s\rangle = h(\tuple z)$, as required.
	\end{itemize}

	Conversely, suppose that $M \models_X \forall u_1 u_2 \exists \tuple z (\tuple t_2 \bowtie \tuple z \wedge (u_1 \not = u_2 \vee \tuple z = \tuple t_1))$, and that therefore there exists a tuple of functions $\tuple H$ such that, for $Y = X[M/u_1 u_2][\tuple H/\tuple z]$, $M \models_Y \tuple t_2 \bowtie \tuple z \wedge (u_1 \not = u_2 \vee \tuple z = \tuple t_1)$. Then consider any assignment $s \in X$, and let $h = s[0/u_1][0/u_2][\tuple H/\tuple z]$. Now, $h \in Y$ and $h(\tuple z) = \tuple t_1\langle s\rangle$; but since $M \models_Y \tuple t_2 \bowtie \tuple z$, this implies that there exists an assignment $h' \in Y$ such that $\tuple t_2 \langle h' \rangle = h(\tuple z) = \tuple t_1\langle s\rangle$. Finally, $h'$ derives from some assignment $s' \in X$, and for this assignment we have that $\tuple t_2\langle s\rangle = \tuple t_2\langle h'\rangle = \tuple t_1\langle s\rangle$ as required.
\end{proof}
As a consequence, inclusion logic is precisely as expressive as equiextension logic: 
\begin{coro}
	Any formula of inclusion logic is equivalent to some formula of equiextension logic, and vice versa.
\end{coro}
\subsection{Exclusion logic}
\label{subsect:exclog}
With the name of \emph{exclusion logic} we refer to (lax, team) first order logic supplemented with the \emph{exclusion atoms} $\tuple t_1 ~|~ \tuple t_2$, with the satisfaction condition given in Definition \ref{def:inc_exc}. \\

As the following results show exclusion logic is, in a very strong sense, equivalent to dependence logic: 
\begin{theo}
	\label{theo:excl_to_dep}
	For all tuples of terms $\tuple t_1$ and $\tuple t_2$, of the same length, there exists a dependence logic formula $\phi$ such that 
	\[
		M \models_X \phi \Leftrightarrow M \models_X \tuple t_1 ~|~ \tuple t_2
	\]
	for all suitable models $M$ and teams $X$.
\end{theo}
\begin{proof}
	This follows immediately from Theorem \ref{theo:DLform}, since the satisfaction condition for the exclusion atom is downwards monotone and expressible in $\Sigma_1^1$. \\

For the sake of completeness, let us write a direct translation of exclusion atoms into dependence logic anyway.

Let $\tuple t_1$ and $\tuple t_2$ be as in our hypothesis, let $\tuple z$ be a tuple of new variables, of the same length of $\tuple t_1$ and $\tuple t_2$, and let $u_1, u_2$ be two further unused variables. Then $M \models_X \tuple t_1 ~|~ \tuple t_2$ if and only if 
\[
	M \models_X \forall \tuple z \exists u_1 u_2 (=\!\!(\tuple z, u_1) \wedge =\!\!(\tuple z, u_2) \wedge ((u_1 = u_2 \wedge \tuple z \not = \tuple t_1) \vee (u_1 \not = u_2 \wedge \tuple z \not = \tuple t_2))).
\]

Indeed, suppose that $M \models_X \tuple t_1 ~|~ \tuple t_2$, let $X' = X[M/\tuple z]$, and let $0, 1$ be two distinct elements in $\dom(M)$. 

Then define the functions $H_1$ and $H_2$ as follows:
\begin{itemize}
	\item For all $s' \in X'$, $H_1(s') = \{0\}$;
	\item For all $s'' \in X'[H_1/u_1]$, $H_2(s'') = \left\{\begin{array}{l l}
		\{0\} &\mbox{ if } s''(\tuple z) \not \in X(\tuple t_1);\\
		\{1\} & \mbox{ if } s''(\tuple z) \in X(\tuple t_1).
\end{array}\right.$
\end{itemize}
Then, for $Y = X'[H_1 H_2/u_1 u_2]$, we have that $M \models_Y =\!\!(\tuple z, u_1)$ and that $M \models_Y =\!\!(\tuple z, u_2)$, since the value of $u_1$ is constant in $Y$ and the value of $u_2$ in $Y$ is functionally determined by the value of $\tuple z$. 

Now split $Y$ into the two subteams $Y_1$ and $Y_2$ defined as 
\begin{align*}
&Y_1 = \{s \in Y : s(u_2) = 0\};\\
&Y_2 = \{s \in Y : s(u_2) = 1\}.
\end{align*}

Clearly, $M \models_{Y_1} u_1 = u_2$ and $M \models_{Y_2} u_1 \not = u_2$; hence, we only need to verify that $M \models_{Y_1} \tuple z \not = \tuple t_1$ and that $M \models_{Y_2} \tuple z \not = \tuple t_2$. 
\\

For the first case, let $h$ be any assignment in $Y_1$: then, by definition, $h(\tuple z) \not = \tuple t_1\langle s\rangle$ for all $s \in X$. But then $h(\tuple z) \not = \tuple t_1\langle h'\rangle$ for all $h' \in Y_1$, and since this is true for all $h \in Y_1$ we have that  $M \models_{Y_1} \tuple z \not = \tuple t_1$, as required.

For the second case, let $h$ be in $Y_2$ instead: then, again by definition, $h(\tuple z) = \tuple t_1\langle s\rangle$ for some $s \in X$. But $M \models_X \tuple t_1 ~|~ \tuple t_2$, and hence $h(\tuple z) \not = \tuple t_2\langle s'\rangle$ for all $s' \in X$; and as in the previous case, this implies that $h(\tuple z) \not = \tuple t_2(h')$ for all $h' \in Y_2$ and, since this argument can be made for all $h \in Y_2$, $M \models_{Y_2} \tuple z \not = \tuple t_2$.\\

Conversely, suppose that
\[
	M \models_X \forall \tuple z \exists u_1 u_2 (=\!\!(\tuple z, u_1) \wedge =\!\!(\tuple z, u_2) \wedge ((u_1 = u_2 \wedge \tuple z \not = \tuple t_1) \vee (u_1 \not = u_2 \wedge \tuple z \not = \tuple t_2))).
\]
Then there exist two functions $H_1$ and $H_2$ such that, for $Y = X[M/\tuple z][H_1 H_2 / u_1 u_2]$,  
\[
	M \models_Y =\!\!(\tuple z, u_1) \wedge =\!\!(\tuple z, u_2) \wedge ((u_1 = u_2 \wedge \tuple z \not = \tuple t_1) \vee (u_1 \not = u_2 \wedge \tuple z \not = \tuple t_2)).
\]

Now, let $s_1$ and $s_2$ be any two assignments in $X$: in order to conclude the proof, I only need to show that $\tuple t_1 \langle s_1\rangle \not = \tuple t_2\langle s_2\rangle$. Suppose instead that $\tuple t_1 \langle s_1 \rangle = \tuple t_2 \langle s_2\rangle = \tuple m$ for some tuple of elements $\tuple m$, and consider two assignments $h_1, h_2$ such that 
\[
h_1 \in \{s_1[\tuple m/\tuple z]\}[H_1 H_2 / u_1 u_2];\footnote{This team and the next one are actually singletons, because $H_1$ and $H_2$ must satisfy the dependency conditions.}
\]
and 
\[
h_2 \in \{s_2[\tuple m/\tuple z]\}[H_1 H_2 / u_1 u_2].
\]
Then $h_1, h_2 \in Y$; and furthermore, since $h_1(\tuple z) = h_2(\tuple z)$ and $M \models =\!\!(\tuple z, u_1) \wedge =\!\!(\tuple z, u_2)$, it must hold that $h_1(\tuple u_1) = h_2(\tuple u_1)$ and $h_1(\tuple u_2) = h_2(\tuple u_2)$.

Moreover, $M \models_Y (u_1 = u_2 \wedge \tuple z \not = \tuple t_1) \vee (u_1 \not = u_2 \wedge \tuple z \not = \tuple t_2)$, and therefore $Y$ can be split into two subteams $Y_1$ and $Y_2$ such that 
\[
	M \models_{Y_1} (u_1 = u_2 \wedge \tuple z \not = \tuple t_1) 
\]
and 
\[
	M \models_{Y_2} (u_1 \not = u_2 \wedge \tuple z \not = \tuple t_2).
\]

Now, as we saw, the assignments $h_1$ and $h_2$ coincide over $u_1$ and $u_2$, and hence either $\{h_1, h_2\} \subseteq Y_1$ or $\{h_1, h_2\} \subseteq Y_2$. But neither case is possible, because 
\[
	h_1(\tuple z) = \tuple m = \tuple t_1\langle s_1\rangle = \tuple t_1\langle h_1\rangle
\]
and therefore $h_1$ cannot be in $Y_1$, and because
\[
	h_2(\tuple z) = \tuple m = \tuple t_2\langle s_2\rangle = \tuple t_2\langle h_2\rangle 
\]
and therefore $h_2$ cannot be in $Y_2$.\\

So we reached a contradiction, and this concludes the proof.

%
%
\end{proof}
\begin{theo}
	\label{theo:dep_to_excl}
	Let $t_1 \ldots t_n$ be terms, and let $z$ be a variable not occurring in any of them. Then the dependence atom $=\!\!(t_1 \ldots t_n)$ is equivalent to the exclusion logic expression 
\[
	\phi = \forall z ( z = t_n \vee (t_1 \ldots t_{n-1} z ~|~ t_1 \ldots t_{n-1} t_n)),
\]
for all suitable models $M$ and teams $X$.
\end{theo} 
\begin{proof}
	Suppose that $M \models_X =\!\!(t_1 \ldots t_n)$, and consider the team $X[M/z]$. Now, let $Y =\{s \in X[M/z] : s(z) = t_n\langle s\rangle\}$ and let $Z = X[M/z] \backslash Y$. 

Clearly, $Y \cup Z = X[M/x]$ and $M \models_Y z = t_n$; hence, if we show that $Z \models t_1 \ldots t_{n-1} z ~|~ t_1 \ldots t_{n-1} t_n$ we can conclude that $M \models_X \phi$, as required. 

Now, consider any two $s, s' \in Z$, and suppose that $t_i\langle s\rangle = t_i\langle s'\rangle$ for all $i = 1 \ldots n-1$. But then $s(z) \not = t_n\langle s'\rangle$: indeed, since $M \models_X =\!\!(t_1 \ldots t_n)$, by the locality of dependence logic and by the downwards closure property we have that $M \models_Z =\!\!(t_1 \ldots t_n)$ and hence that $t_n\langle s\rangle = t_n\langle s'\rangle$. 

Therefore, if we had that $s(z) = t_n\langle s'\rangle$, it would follow that $s(z) = t_n\langle s'\rangle = t_n\langle s\rangle$ and $s$ would be in $Y$ instead. 

So $s(z) \not = t_n\langle s'\rangle$, and since this holds for all $s$ and $s'$ in $Z$ which coincide over $t_1 \ldots t_{n-1}$ we have that 
\[
	M \models_Z t_1 \ldots t_{n-1} z ~|~ t_1 \ldots t_{n-1}t_n,
\]
as required.\\

Conversely, suppose that $M \models_X \phi$, and let $s, s' \in X$ assign the same values to $t_1 \ldots t_{n-1}$. Now, by the definition of $\phi$, $X[M/z]$ can be split into two subteams $Y$ and $Z$ such that $M \models_Y z = t_n$ and\\ $M \models_Z (t_1 \ldots t_{n-1} z ~|~ t_1 \ldots t_{n-1} t_n)$.

Now, suppose that $t_n\langle s\rangle = m$ and $t_n \langle s'\rangle = m'$, and that $m \not = m'$: then $s[m'/z]$ and $s'[m/z]$ are in $s[M/z]$ but not in $Y$, and hence they are both in $Z$. But then, since $\tuple t_i\langle s\rangle = \tuple t_i\langle s'\rangle$ for all $i = 1 \ldots n-1$, 
\[
t_n\langle s' \rangle = m' = s[m'/z](z) \not = t_n\langle s'[m/z]\rangle = t_n\langle s'\rangle
\]
which is a contradiction. Therefore, $m = m'$, as required.
\end{proof}
\begin{coro}
\label{coro:DLeqEL}
Dependence logic is precisely as expressive as exclusion logic, both with respect to definability of sets of teams and with respect to sentences.
\end{coro}
\subsection{Inclusion/exclusion logic}
\label{subsect:ielogic}
Now that we have some information about inclusion logic and about exclusion logic, let us study \emph{inclusion/exclusion logic} (I/E logic for short), that is, the formalism obtained by adding both inclusion and exclusion atoms to the language of first-order logic. \\

By the results of the previous sections, we already know that inclusion atoms are expressible in independence logic and that exclusion atoms are expressible in dependence logic; furthermore, by Theorem \ref{theo:DL2IL}, dependence atoms are expressible in independence logic.

Then it follows at once that I/E logic is contained in independence logic: 
\begin{coro}
	For every inclusion/exclusion logic formula $\phi$ there exists an independence logic formula $\phi^*$ such that 
	\[
		M \models_X \phi \Leftrightarrow M \models_X \phi^*
	\]
	for all suitable models $M$ and teams $X$. 
\end{coro}

Now, is I/E logic properly contained in independence logic? \\

As the following result illustrates, this is not the case: 
\begin{theo}
	Let $\indep{\tuple t_1}{\tuple t_2}{\tuple t_3}$ be an independence atom, and let $\phi$ be the formula 
\begin{align*}
	& \forall \tuple p \tuple q \tuple r ~ \exists u_1 u_2 u_3 u_4 
	\left( \bigwedge_{i=1}^4 =\!\!(\tuple p \tuple q \tuple r, u_i) \wedge 
((u_1 \not = u_2 \wedge (\tuple p \tuple q ~|~ \tuple t_1 \tuple t_2))
\vee\right.\\
&~ ~\left. \vee (u_1 = u_2 \wedge u_3 \not= u_4 \wedge (\tuple p \tuple r ~|~ \tuple t_1 \tuple t_3)) \vee (u_1 = u_2 \wedge u_3 = u_4 \wedge (\tuple p \tuple q \tuple r \subseteq \tuple t_1 \tuple t_2 \tuple t_3))) \right)
\end{align*}
where the dependence atoms are used as shorthands for the corresponding exclusion logic expressions, which exist because of Theorem \ref{theo:dep_to_excl}, and where all the quantified variables are new. 

Then, for all suitable models $M$ and teams $X$, 
\[
	M \models_X \indep{\tuple t_1}{\tuple t_2}{\tuple t_3} \Leftrightarrow M \models_X \phi.
\]
\end{theo}
\begin{proof}
Suppose that $M \models_X \indep{\tuple t_1}{\tuple t_2}{\tuple t_3}$, and consider the team $X' = X[M/\tuple p \tuple q \tuple r]$. 

Now, let $0$ and $1$ be two distinct elements of the domain of $M$, and let the functions $F_1 \ldots F_4$ be defined as follows:
\begin{itemize}
\item For all $s \in X'$, $F_1(s) = 0$; 
\item For all $s \in X'[F_1/u_1]$, 
	\[
		F_2(s) = \left\{\begin{array}{l l}
			0 & \mbox{ if there exists a } s' \in X \mbox{ such that } \tuple t_1 \langle s'\rangle \tuple t_2 \langle s'\rangle = s(\tuple p) s(\tuple q);\\
			1 & \mbox{ otherwise;}
		\end{array}
		\right.
	\]
\item For all $s \in X'[F_1/u_1][F_2/u_2]$, $F_3(s) = 0$; 
\item For all $s \in X'[F_1/u_1][F_2/u_2][F_3/u_3]$, 
	\[
		F_4(s) = \left\{\begin{array}{l l}
			0 & \mbox{ if there exists a } s' \in X \mbox{ such that } \tuple t_1 \langle s'\rangle \tuple t_3 \langle s'\rangle = s(\tuple p) s(\tuple r);\\
			1 & \mbox{ otherwise.}
		\end{array}
		\right.
	\]
\end{itemize}
Now, let $Y = X'[F_1/u_1][F_2/u_2][F_3/u_3][F_4/u_4]$: by the definitions of $F_1 \ldots F_4$, it holds that all dependencies are respected. Let then $Y$ be split into $Y_1$, $Y_2$ and $Y_3$ according to: 
\begin{itemize}
\item $Y_1 = \{s \in Y : s(u_1) \not = s(u_2)\}$; 
\item $Y_2 = \{s \in Y : s(u_3) \not = s(u_4)\} \backslash Y_1$;
\item $Y_3 = Y \backslash (Y_1 \cup Y_2)$.
\end{itemize}

Now, let $s$ be any assignment of $Y_1$: then, since $s(u_1) \not = s(u_2)$, by the definitions of $F_1$ and $F_2$ we have that 
\[
	\forall s' \in Y, s(\tuple p) s(\tuple q) \not = \tuple t_1\langle s'\rangle \tuple t_2\langle s'\rangle
\] 
and, in particular, that the same holds for all the $s' \in Y_1$. Hence, 
\[
	M \models_{Y_1}  u_1 \not = u_2 \wedge (\tuple p \tuple q ~|~ \tuple t_1 \tuple t_2),
\] 
as required. 

Analogously, let $s$ be any assignment of $Y_2$: then $s(u_1) = s(u_2)$, since otherwise $s$ would be in $Y_1$, $s(u_3) \not = s(u_4)$ and 
\[
	\forall s' \in Y, s(\tuple p) s(\tuple r) \not = \tuple t_1\langle s'\rangle \tuple t_3\langle s'\rangle
\] 
and therefore 
\[
	M \models_{Y_2} u_1 = u_2 \wedge u_3 \not= u_4 \wedge (\tuple p \tuple r ~|~ \tuple t_1 \tuple t_3).
\]
Finally, suppose that $s \in Y_3$: then, by definition, $s(u_1) = s(u_2)$ and $s(u_3) = s(u_4)$. Therefore, there exist two assignments $s'$ and $s''$ in $X$ such that 
\[
	\tuple t_1\langle s'\rangle \tuple t_2\langle s'\rangle = s(\tuple p) s(\tuple q)
\]
and 
\[
	\tuple t_1\langle s''\rangle \tuple t_3\langle s''\rangle = s(\tuple p) s(\tuple r)
\]
But by hypothesis we know that $M \models_X \indep {\tuple t_1} {\tuple t_2}{\tuple t_3}$, and $s'$ and $s''$ coincide over $\tuple t_1$, and therefore there exists a new assignment $h \in X$ such that 
\[
	\tuple t_1\langle h\rangle \tuple t_2\langle h\rangle \tuple t_3\langle h\rangle = s(\tuple p) s(\tuple q) s(\tuple r). 
\]

Now, let $o$ be the assignment of $Y$ given by 
\[
	o = h[\tuple t_1\langle h\rangle \tuple t_2\langle h\rangle \tuple t_3\langle h\rangle / \tuple p \tuple q \tuple r][F_1 \ldots F_4 /u_1 \ldots u_4]:
\]
by the definitions of $F_1 \ldots F_4$ and by the construction of $o$, we then get that 
\[
	o(u_1) = o(u_2) = o(u_3) = o(u_4) = 0
\]
and therefore that $o \in Y_3$. 

But by construction, 
\[
	\tuple t_1\langle o\rangle \tuple t_2\langle o\rangle \tuple t_3\langle o\rangle = \tuple t_1\langle h\rangle \tuple t_2\langle h\rangle \tuple t_3\langle h\rangle = s(\tuple p) s(\tuple q) s(\tuple r),
\]
and hence
\[
M \models_{Y_3} \tuple p \tuple q \tuple r \subseteq \tuple t_1 \tuple t_2 \tuple t_3
\]
as required.\\

Conversely, suppose that $M \models_X \phi$, and let $s, s' \in X$ be such that $\tuple t_1 \langle s\rangle = \tuple t_1 \langle s'\rangle$. Now, consider the two assignments $h, h' \in X' = X[M/\tuple p \tuple q \tuple r]$ given by 
\[
	h = s[\tuple t_1\langle s\rangle/\tuple p][\tuple t_2\langle s\rangle/\tuple q][\tuple t_3\langle s'\rangle/\tuple r]
\]
and 
\[
	h' = s'[\tuple t_1\langle s\rangle/\tuple p][\tuple t_2\langle s\rangle/\tuple q][\tuple t_3\langle s'\rangle/\tuple r].
\]

Now, since $M \models_X \phi$, there exist functions $F_1 \ldots F_4$, depending only on $\tuple p$, $\tuple q$ and $\tuple r$, such that\\$Y = X'[F_1/u_1][F_2/u_2][F_3/u_3][F_4/u_4]$ can be split into three subteams $Y_1$, $Y_2$ and $Y_3$ and
\begin{align*}
&M \models_{Y_1} (u_1 \not = u_2 \wedge (\tuple p \tuple q ~|~ \tuple t_1 \tuple t_2));\\
&M \models_{Y_2} (u_1 = u_2 \wedge u_3 \not= u_4 \wedge (\tuple p \tuple r ~|~ \tuple t_1 \tuple t_3));\\
&M \models_{Y_3} (u_1 = u_2 \wedge u_3 = u_4 \wedge (\tuple p \tuple q \tuple r \subseteq \tuple t_1 \tuple t_2 \tuple t_3)).
\end{align*}
Now, let 
\[
	o = h[F_1/u_1][F_2/u_2][F_3/u_3][F_4/u_4]
\]
and
\[
	o' = h'[F_1/u_1][F_2/u_2][F_3/u_3][F_4/u_4]:
\]
since the $F_i$ depend only on $\tuple p \tuple q \tuple r$ and the values of these variables are the same for $h$ and for $h'$, we have that $o$ and $o'$ have the same values for $u_1 \ldots u_4$, and therefore that they belong to the same $Y_i$. 

But they cannot be in $Y_1$ nor in $Y_2$, since
\[
	o(\tuple p)o(\tuple q) = o'(\tuple p)o'(\tuple q) = \tuple t_1\langle s\rangle \tuple t_2\langle s\rangle = \tuple t_1\langle o\rangle \tuple t_2\langle o\rangle
\]
and 
\[
	o(\tuple p)o(\tuple r) = o'(\tuple p)o'(\tuple r) = \tuple t_1\langle s'\rangle \tuple t_3\langle s'\rangle = \tuple t_1\langle o'\rangle \tuple t_3\langle o'\rangle;
\]
therefore, $o$ and $o'$ are in $Y_3$, and there exists an assignment $o'' \in Y_3$ with 
\[
	\tuple t_1\langle o''\rangle \tuple t_2\langle o''\rangle \tuple t_3\langle o''\rangle = o(\tuple p) o(\tuple q) o(\tuple r) = \tuple t_1\langle s\rangle\tuple t_2\langle s\rangle \tuple t_3\langle s'\rangle
\]
and, finally, there exists a $s'' \in X$ such that $\tuple t_1\langle s''\rangle \tuple t_2\langle s''\rangle \tuple t_3\langle s''\rangle = \tuple t_1\langle s\rangle\tuple t_2\langle s\rangle \tuple t_3\langle s'\rangle$,  as required. 
\end{proof}
Independence logic and inclusion/exclusion logic are therefore equivalent: 
\begin{coro}
	\label{coro:IL_eq_IE}
	Any independence logic formula is equivalent to some inclusion/exclusion logic formula, and any inclusion/exclusion logic formula is equivalent to some independence logic formula. 
\end{coro}

Figure \ref{fig:logrel2} summarizes the translatability\footnote{To be more accurate, Figure \ref{fig:logrel2} represents the translatability relations between the logics which we considered, \emph{with respect to all formulas}. Considering sentences only would lead to a different graph.} relations between the logics of imperfect information which have been considered in this work.
\begin{figure}
\begin{center}
	\epsfig{file=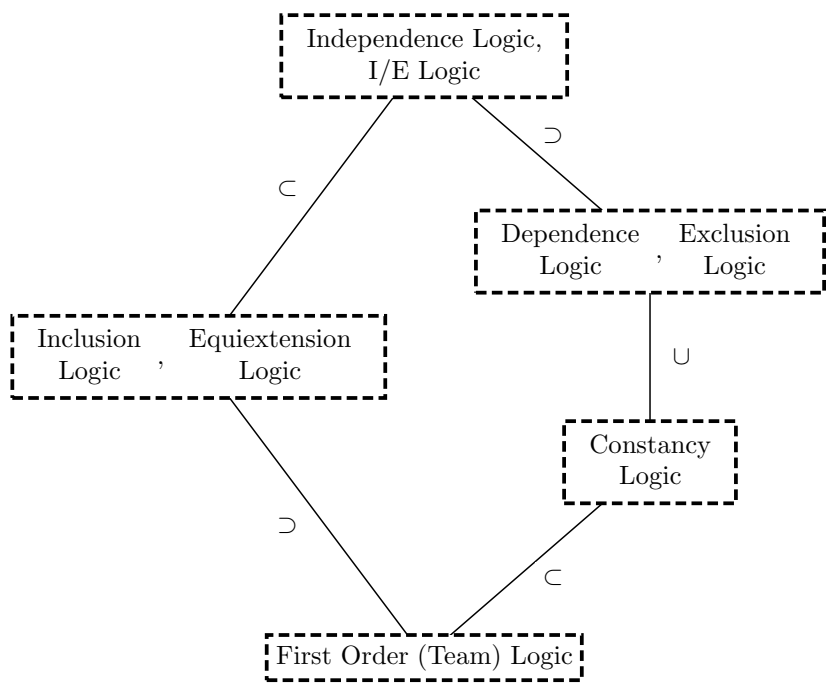}
\end{center}
 	\caption{Relations between logics of imperfect information (wrt formulas)}
	\label{fig:logrel2}
\end{figure}

Let us finish this section verifying that I/E logic (and, as a consequence, also inclusion logic, equiextension logic and independence logic) with the lax semantics is local: 
\begin{theo}[Inclusion/exclusion logic with lax semantics is local]
	\label{theo:ielocal}
	Let $M$ be a first order model, let $\phi$ be any I/E logic formula and let $V$ be a set of variables such that $\free(\phi) \subseteq V$. Then, for all suitable teams $X$,
	\[
	M \models_X \phi \Leftrightarrow M \models_{X_{\upharpoonright V}} \phi
	\]
\end{theo}
\begin{proof}
	The proof is by structural induction on $\phi$. 
	\begin{enumerate}
		\item If $\phi$ is a first order literal, an inclusion atom or an exclusion atom then the statement follows trivially from the corresponding semantic rule;
		\item Let $\phi$ be of the form $\psi \vee \theta$, and suppose that $M \models_X \psi \vee \theta$. Then, by definition, $X = Y \cup Z$ for two subteams $Y$ and $Z$ such that $M \models_Y \psi$ and $M \models_Z \theta$. Then, by induction hypothesis, $M \models_{Y_{\upharpoonright V}} \psi$ and $M \models_{Z_{\upharpoonright V}} \theta$. But $X_{\upharpoonright V} = Y_{\upharpoonright V} \cup Z_{\upharpoonright V}$: indeed, $s \in X$ if and only if $s \in Y$ or $s \in Z$, and hence $s_{\upharpoonright V} \in X_{\upharpoonright V}$ if and only if it is in $Y_{\upharpoonright V}$ or in $Z_{\upharpoonright V}$. Hence, $M \models_{X_{\upharpoonright V}} \psi \vee \theta$, as required. 

			Conversely, suppose that $M \models_{X_{\upharpoonright V}} \psi \vee \theta$, that is, that $X_{\upharpoonright V} = Y' \cup Z'$ for two subteams $Y'$ and $Z'$ such that $M \models_{Y'} \psi$ and $M \models_{Z'} \theta$. Then define $Y = \{s \in X : s_{\upharpoonright V} \in Y'\}$ and $Z = \{s \in X : s_{\upharpoonright V} \in Z'\}$. Now, $X = Y \cup Z$: indeed, if $s \in X$ then $s_{\upharpoonright V}$ is in $X_{\upharpoonright V}$, and hence it is in $Y'$ or in $Z'$, and on the other hand if $s$ is in $Y$ or in $Z$ then it is in $X$ by definition. Furthermore, $Y_{\upharpoonright V} = Y'$ and $Z_{\upharpoonright V} = Z'$,\footnote{By definition, $Y_{\upharpoonright V} \subseteq Y'$ and $Z_{\upharpoonright V} \subseteq Z'$. On the other hand, if $s' \in Y'$ then $s' \in X_{\upharpoonright V}$, and hence $s'$ is of the form $s_{\upharpoonright V}$ for some $s \in X$, and therefore this $s$ is in $Y$ too, and finally $s' = s_{\upharpoonright V} \in Y_{\upharpoonright V}$. The same argument shows that $Z' \subseteq Z_{\upharpoonright V}$.}  and hence by induction hypothesis $M \models_Y \psi$ and $M \models_Z \theta$, and finally $M \models_X \psi \vee \theta$. 
		\item Let $\phi$ be of the form $\psi \wedge \theta$. Then $M \models_X \psi \wedge \theta$ if and only if $M \models_X \psi$ and $M \models_X \theta$, that is, by induction hypothesis, if and only if $M \models_{X_{\upharpoonright V}} \psi$ and $M \models_{X_{\upharpoonright V}} \theta$. But this is the case if and only if $M \models_{X_{\upharpoonright V}} \psi \wedge \theta$, as required. 
		\item Let $\phi$ be of the form $\exists x \psi$, and suppose that $M \models_X \exists x \psi$. Then there exists a function $H: X \rightarrow \part(\dom(M)) \backslash \{\emptyset\}$ such that $M \models_{X[H/x]} \psi$. Then, by induction hypothesis, $M \models_{(X[H/x])_{\upharpoonright V \cup \{x\}}} \psi$. \\

			Now consider the function $H': X_{\upharpoonright V} \rightarrow \part(\dom(M)) \backslash \emptyset$ which assigns to every $s' \in X_{\upharpoonright V}$ the set
				\[
					H'(s') = \bigcup \{H(s) : s \in X, s' = s_{\upharpoonright V}\}.
				\]
				Then $H'$ assigns a nonempty set to every $s' \in X_{\upharpoonright V}$, as required; and furthermore, $X_{\upharpoonright V}[H'/x]$ is precisely $(X[H/x])_{\upharpoonright V \cup \{x\}}$.\footnote{Indeed, suppose that $s' \in X[H/x]$: then there exists a $s \in X$ such that $s' = s[m/x]$ for some $m \in H(s)$. Then $s_{\upharpoonright V} \in X_{\upharpoonright V}$, and moreover $m \in H'(s_{\upharpoonright V})$ by the definition of $H'$, and hence $s'_{\upharpoonright V \cup \{x\}} = s_{\upharpoonright V}[m/x] \in X_{\upharpoonright V}[H'/x]$.

				Conversely, suppose that $h' \in X_{\upharpoonright V}[H'/x]$: then there exists a $h \in X_{\upharpoonright V}$ such that $h' = h[m/x]$ for some $m \in H'(h)$. But then there exists a $s \in X$ such that $h = s_{\upharpoonright V}$ and such that $m \in H(s)$; and therefore, $s[m/x] \in X[H/x]$, and finally $h' = h[m/x] = (s[m/x])_{\upharpoonright V \cup \{x\}} \in (X[H/x])_{\upharpoonright V \cup \{x\}}$.} Therefore, $M \models_{X_{\upharpoonright V}} \exists x \psi$, as required.

				Conversely, suppose that $M \models_{X_{\upharpoonright V}} \exists x \psi$, that is, that $M \models_{X_{\upharpoonright V}[H'/x]} \psi$ for some $H'$. Then define the function $H: X \rightarrow \part(\dom(M)) \backslash \{x\}$ so that $H(s) = H'(s_{\upharpoonright V})$ for all $s \in X$; now, $X_{\upharpoonright V}[H'/x] = (X[H/x])_{\upharpoonright V \cup \{x\}}$,\footnote{In brief, for all $s \in X$ and all $m \in \dom(M)$ we have that $m \in H'(s_{\upharpoonright V})$ if and only if $m \in H(s)$, by definition. Hence, for all such $s$ and $m$, $s_{\upharpoonright V}[m/x] \in X_{\upharpoonright V}[H'/x]$ if and only if $s[m/x] \in X[H/x]$.} and hence by induction hypothesis $M \models_X \exists x \psi$. 
			\item For all suitable teams $X$, $X[M/x]_{\upharpoonright V \cup \{x\}} = X_{\upharpoonright V}[M/x]$; and hence, $M \models_{X_{\upharpoonright V}} \forall x \psi \Leftrightarrow M \models_{X[M/x]_{\upharpoonright V \cup \{x\}}} \psi \Leftrightarrow M \models_{X[M/x]} \psi \Leftrightarrow M \models_X \forall x \psi$, as required.
	\end{enumerate}
\end{proof}
\section{Game theoretic semantics}
\label{sect:gamesem}
By this point, we have developed a team semantics for inclusion/exclusion logic and we have examined the relations between it and other logics of imperfect information. In this section, an equivalent game theoretic semantics for inclusion/exclusion logic will be developed; once this is done, the semantics for inclusion logic and for exclusion logic will simply be the restrictions of this semantics to the corresponding sublanguages. The connection between game semantics and team semantics, moreover, will allow us to revisit and further justify the distinction between lax and strict connectives introduced in Section \ref{subsect:laxstrict}. However, we will not discuss here the history or the motivations of game theoretic semantics, nor its connections to other game-theoretical approaches to formal semantics. The interested reader is referred to \cite{hintikka83} and \cite{hintikkasandu97} for a more philosophically oriented discussion of game theoretic semantics; in the rest of this section, we will content ourselves to present such a semantics for the case of I/E logic and prove its equivalence to team semantics. 
\begin{defin}[Semantic games for I/E logic]
	Let $\phi$ be an I/E logic formula, let $M$ be a first order model over a signature containing that of $\phi$ and let $X$ be a team over $M$ whose domain contains all free variables of $\phi$. Then the game $G^M_X(\phi)$ is defined as follows:
	\begin{itemize}
		\item There are two players, called $I$ and $II$;\footnote{These players can also be named Falsifier and Verifier, or Abelard and Eloise.}
		\item The positions of the game are expressions of the form $(\psi, s)$, where $\psi$ is an instance of a subformula of $\phi$ and $s$ is an assignment whose domain contains all free variables of $\psi$;
		\item The initial positions are all those of the form $(\phi, s)$ for $s \in X$;
		\item The terminal positions are those of the form $(\alpha, s)$, where $\alpha$ is a first order literal, an inclusion atom, or an exclusion atom;
		\item If $p = (\psi, s)$ is not a terminal position, the set $S(p)$ of its \emph{successors} is defined according to the following rules:
			\begin{enumerate}
				\item If $\psi$ is of the form $\theta_1 \vee \theta_2$ or $\theta_1 \wedge \theta_2$ then $S(p) = \{(\theta_1, s), (\theta_2, s)\}$;
				\item If $\psi$ is of the form $\exists x \theta$ or $\forall x \theta$ then $S(p) = \{(\theta, s[m/x]) : m \in \dom(M)\}$;
			\end{enumerate}
		\item If $p = (\psi, s)$ is not a terminal position, the \emph{active player} $T(p) \in \{I, II\}$ is defined according to the following rules:
			\begin{enumerate}
				\item If $\psi$ is of the form $\theta_1 \vee \theta_2$ or $\exists x \theta$ then $T(p) = II$;
				\item If $\psi$ is of the form $\theta_1 \wedge \theta_2$ or $\forall x \theta$ then $T(p) = I$;
			\end{enumerate}
		\item A terminal position $p = (\alpha, s)$ is \emph{winning} for Player $II$ if and only if 
			\begin{itemize}
				\item $\alpha$ is a first order literal and $M \models_s \alpha$ in the usual first order sense, or 
				\item $\alpha$ is an inclusion atom $\tuple t_1 \subseteq \tuple t_2$ and $s$ is any assignment, or 
				\item $\alpha$ is an exclusion atom $\tuple t_1 ~|~ \tuple t_2$ and $s$ is any assignment.
			\end{itemize}
			If a terminal position is not winning for Player $II$, it is winning for Player $I$. 
	\end{itemize}
\end{defin}
The definitions of \emph{play}, \emph{complete play} and \emph{winning play} are straightforward:
\begin{defin}
	Let $G^M_X(\phi)$ be a semantic game as above. Then a \emph{play} for $G^M_X(\phi)$ is a finite sequence of positions $p_1 \ldots p_n$ such that 
	\begin{itemize}
		\item $p_1$ is an initial position;
		\item For all $i = 2 \ldots n$, $p_i \in S(p_{i-1})$.
	\end{itemize}

	Such a play is said to be \emph{complete} if, furthermore, $p_n$ is a terminal position; and it is \emph{winning} for Player $II$ [$I$] if and only if $p_n$ is a winning position for $II$ [$I$].
\end{defin}

However, it will be useful to consider \emph{non-deterministic strategies} rather than deterministic ones only:
\begin{defin}
	Let $G^M_X(\phi)$ be a semantic game as above. Then a \emph{strategy} for Player $II$ $[I]$ in $G^M_X(\phi)$ is a function $\tau$ sending each position $p = (\psi, s)$ with $T(p) = II$ $[I]$ into some $\tau(p) \in \mathcal P(S(p)) \backslash \emptyset$.\\

	Such a strategy is said to be \emph{deterministic} if, for all such $p$, $\tau(p)$ is a singleton. 

	A play $p_1 \ldots p_n$ is said to \emph{follow} a strategy $\tau$ for $II$ [$I$] if and only if, for all $i \in 1 \ldots n-1$, 
	\[
		T(p_i) = II~[I] \Rightarrow p_{i+1} \in \tau(p_i).
	\]

	A strategy $\tau$ for $II$ $[I]$ is \emph{winning} for Player $II$ [$I$] if and only if all complete plays $\tuple p$ which follow $\tau$ are winning for $II$ [$I$]. \\

	The set of all plays of $G^M_X(\phi)$ in which Player $\rho \in \{I, II\}$ follows strategy $\tau$ will be written as $P(G^M_X(\phi), \rho:\tau)$.
\end{defin}
So far, inclusion and exclusion atoms play little role in our semantics, as they always correspond to winning positions for Player $II$. Similarly to dependence atoms in \cite{vaananen07}, however, inclusion and exclusion atoms restrict the \emph{set of strategies} available to Player $II$. This is modeled by the following definition of \emph{uniform strategy}:
\begin{defin}
	Let $G^M_X(\phi)$ be a semantic game as above. Then a strategy $\tau$ for Player $II$ is said to be \emph{uniform} if and only if, for all complete plays $p_1 \ldots p_n = \tuple p \in P(G^M_X(\phi), II:\tau)$, 
	\begin{enumerate}
		\item If $p_n$ is of the form $p_n = (\tuple t_1 \subseteq \tuple t_2, s)$ then there exists a play $q_1 \ldots q_{n'} \in P(G^M_X(\phi), II:\tau)$ such that $q_{n'} = (\tuple t_1 \subseteq \tuple t_2, s')$ for the same instance of the inclusion atom and such that $\tuple t_2 \langle s'\rangle = \tuple t_1\langle s\rangle$;
		\item If $p_n$ is of the form $p_n = (\tuple t_1 ~|~ \tuple t_2, s)$ then for all plays $q_1 \ldots q_{n'} \in P(G^M_X(\phi), II:\tau)$ such that $q_{n'} = (\tuple t_1 ~|~ \tuple t_2, s')$ for the same instance of the exclusion atom it holds that $\tuple t_1\langle s\rangle \not = \tuple t_2 \langle s'\rangle$.
	\end{enumerate}
\end{defin}
This notion of uniformity also makes it clear why in inclusion logic there is a difference between working with non-deterministic and with deterministic strategies: whereas the uniformity condition for dependence atoms restrict the information available to Player $II$ thorough the game, the one for inclusion atoms requires that the set of \emph{possible plays}, given a strategy for Player $II$, is \emph{closed} with respect to certain monotonically increasing operators. This phenomenon does not occur for the uniformity conditions of exclusion atoms, whose form is more similar to the conditions of the dependence atom in \cite{vaananen07}. The next definition and the lemmas following it will be of some use in order to prove the main result of this section: 
\begin{defin}
	Let $G^M_X(\phi)$ be a game as in our previous definitions and let $\tau$ be a strategy for Player $II$ in it. Furthermore, let $\psi$ be an instance of a subformula of $\phi$ and let 
	\[
		Y = \{s : \mbox{ there is a play in } P(G^M_X, II: \tau) \mbox{ passing through } (\psi, s)\}.
	\]
	Furthermore, let $\tau'$ be the restriction of $\tau$ to $G^M_Y(\psi)$, in the sense that $\tau'(\theta, s) = \tau(\theta, s)$ for all $\theta$ contained in $\psi$ and for all assignments $s$. Then we say that $(Y, \psi, \tau')$ is a $M$-\emph{successor} of $(X, \phi, \tau)$, and we write 
	\[
		(Y, \psi, \tau') \leq_M (X, \phi, \tau).
	\]
\end{defin}
From a game-theoretical perspective, the notion of $M$-successor can be seen as a generalization of the notion of the concepts of \emph{subgame} and \emph{substrategy} to multiple initial positions and to games of imperfect information.

\begin{lemma}
	\label{lemma:MSucch_Uniq}
	Let $G^M_X(\phi)$ be a semantic game for I/E logic, and let $\psi$ be an instance of a subformula in $\phi$. Then there exists precisely one team $Y$ and precisely one strategy $\tau'$ for $G^M_Y(\psi)$ such that $(Y, \psi, \tau') \leq_M (X, \phi, \tau)$. 
\end{lemma}
\begin{proof}
	Obvious from definition.
\end{proof}
\begin{lemma}
	\label{lemma:MSucc_Play}
	Let $G^M_X(\phi)$ be a semantic game as usual, and let $\tau$ be a strategy for Player $II$ in it. Furthermore, let $\psi$ be an instance of a subformula of $\phi$ and let $Y$, $\tau'$ be such that $(Y, \psi, \tau') \leq_M (X, \phi, \tau)$. 
	Then 
	\begin{enumerate}
		\item For any play $p_1 \ldots p_n = \tuple p \in P(G^M_X(\phi), II:\tau)$ passing through the subformula $\psi$ there exist a $k \in 1 \ldots n$ such that $p_k \ldots p_n$ is a play in $P(G^M_Y(\psi), II:\tau')$; 
		\item For any play $q_1 \ldots q_m = \tuple q \in P(G^M_Y(\psi), II:\tau')$ there exists a $k \in 1 \ldots n$ and positions $p_1 \ldots p_{k}$ of the game $G^M_X(\phi)$ such that $p_1 \ldots p_k q_1 \ldots q_m$ is a play in $(G^M_X(\psi), II:\tau)$.
	\end{enumerate}
\end{lemma}
\begin{proof}
	\begin{enumerate}
		\item Consider any play $p_1 \ldots p_n$ as in our hypothesis, and let $k \in 1 \ldots n$ be such that $p_k = (\psi, s)$ for some assignment $s$. Then, by definition of $M$-successors, $s \in Y$ and $p_k$ is a possible initial position of $G^M_Y(\psi)$; furthermore, again by the definition of $M$-successor, we have that, for all $i = k \ldots n-1$, $\tau'(p_i) = \tau(p_i) \ni p_{i+1}$. \\

			Hence, $p_k \ldots p_n$ is a play in $P(G^M_Y(\psi), II:\tau')$, as required.
		\item Consider any play $q_1 \ldots q_m$ as in our hypothesis, and hence let $q_1 = (\psi, s)$ for some $s \in Y$. Then, by definition, there exists a play $p_1 \ldots p_n$ in $P(G^M_X(\psi), II:\tau)$ such that $p_{k+1} = q_1 = (\psi, s)$ for some $k \in 0 \ldots n-1$. But $\tau'$ behaves like $\tau$, and hence $\tau(q_i) = \tau'(q_i) \ni q_{i+1}$ for all $i = 1 \ldots m-1$. Thus, $p_1 \ldots p_{k}q_1 \ldots q_m$ is a play in $(G^M_X(\psi), II:\tau)$, as required.
	\end{enumerate}
\end{proof}
\begin{lemma}
	\label{lemma:MSucc_Sub}
	Let $G^M_X(\phi)$ be a semantic game as usual, and let $\tau$ be a strategy for Player $II$ in it. Furthermore, let $\psi$ be an instance of a subformula of $\phi$ and let $Y$, $\tau'$ be such that $(Y, \psi, \tau') \leq_M (X, \phi, \tau)$.\\

	Then 
	\begin{enumerate}
		\item If $\tau$ is winning for $II$ in $G^M_X(\phi)$ then $\tau'$ is winning for $II$ in $G^M_Y(\psi)$;
		\item If $\tau$ is uniform in $G^M_X(\phi)$ then $\tau'$ is uniform in $G^M_Y(\psi)$;
		\item If $\tau$ is deterministic in $G^M_X(\phi)$ then $\tau'$ is deterministic in $G^M_Y(\psi)$. 
	\end{enumerate}
\end{lemma}
\begin{proof}
	\begin{enumerate}
		\item Suppose that $\tau$ is winning\footnote{Here and in the rest of the work, when we write ``winning'' without specifying the player we mean ``winning for Player $II$''.}, and consider any play $q_1 \ldots q_m = \tuple q \in P(G^M_Y(\psi), II: \tau')$. Then, by Lemma \ref{lemma:MSucc_Play}, there exists a play $p_1 \ldots p_n \in P(G^M_X(\phi), II:\tau)$ such that $p_k \ldots p_n = q_1 \ldots q_m$ for some $k \in 1 \ldots m$. But $\tau$ is a winning strategy for $II$ in $G^M_X(\phi)$ and therefore $p_n$ is a winning position, as required.
		\item Suppose that $\tau$ is uniform, and consider any play $q_1 \ldots q_m = \tuple q \in P(G^M_Y(\psi), II: \tau')$.\\

			Then, again, there exists a play  $p_1 \ldots p_n = \tuple p \in P(G^M_X(\psi), II:\tau)$ such that $p_k \ldots p_n = q_1 \ldots q_m$ for some $k$.\\

			Now suppose that $p_n = q_m = (\tuple t_1 \subseteq \tuple t_2, s)$: then, since $\tau$ is a uniform strategy, there exists another play $p'_1 \ldots p'_{n'}$ in $(G^M_X(\phi), II:\tau)$ such that $p'_{n'} = (\tuple t_1 \subseteq \tuple t_2, s')$ for the same instance of the inclusion atom and for a $s'$ such that $t_2 \langle s'\rangle = t_1 \langle s\rangle$.\\

			Since $p_n$ and $p'_{n'}$ correspond the same dependency atom of $\tuple p$, it must be the case that the play $p'_1 \ldots p'_{n'}$ passes through $\psi$; and therefore, by Lemma \ref{lemma:MSucc_Play}, there exists some $j \in 1 \ldots n'$ such that $p'_j \ldots p'_{n'}$ is a play in $P(G^M_Y(\psi), \tau')$, thus satisfying the uniformity condition for $\tau'$. \\

			Now suppose that $p_n = q_m = (\tuple t_1 ~|~ \tuple t_2, s)$ instead, and consider any other play $q'_1 \ldots q'_{m'} \in P(G^M_Y(\psi), \tau')$ such that $q'_m = (\tuple t_1 ~|~ \tuple t_2, s')$ for the same instance of the exclusion atom. Then there exist positions $p'_1 \ldots p'_{k'}$ such that $p'_1 \ldots p'_{k'}q'_1 \ldots q'_{m'}$ is a play in $P(G^M_X(\phi), II:\tau)$. But $\tau$ is uniform, and therefore $s(\tuple t_1) \not = s'(\tuple t_2)$, as required.
	\item This follows trivially by the definition of $M$-successor.
	\end{enumerate}
\end{proof}
\begin{lemma}
	\label{lemma:MSucc_Sup}
	Let $G^M_X(\phi)$ be a semantic game for I/E logic and let $\tau$ be a strategy for $II$ in it. Furthermore, let $\psi_1 \ldots \psi_t$ be an enumeration of all immediate subformulas of $\phi$, and let $Y_1 \ldots Y_t$, $\tau_1 \ldots \tau_t$ be such that $(Y_i, \psi_i, \tau_i) \leq_M (X, \phi, \tau)$ for all $i \in 1 \ldots t$. Then 
	\begin{enumerate}
		\item If all $\tau_i$ are winning in $G^M_{Y_i}(\psi_i)$ then $\tau$ is winning in $G^M_X(\phi)$; 
		\item If all $\tau_i$ are uniform in $G^M_{Y_i}(\psi_i)$ then $\tau$ is uniform in $G^M_X(\phi)$;
		\item If all $\tau_i$ are deterministic in $G^M_{Y_i}(\psi_i)$ and $T(\phi) = I$\footnote{With a slight abuse of notation, we say that $T(\psi) = \alpha$ if $T(\psi, s) = \alpha$ for all suitable assignments $s$. In other words, $T(\psi) = I$ if $\psi$ is of the form $\psi_1 \wedge \psi_2$ or of the form $\forall v  \psi_1$, and $T(\psi) = II$ if $\psi$ is of the form $\psi_1 \vee \psi_2$ or $\exists v \psi_1$.} then $\tau$ is deterministic;
		\item If all $\tau_i$ are deterministic in $G^M_{Y_i}(\psi_i)$, $T(\phi) = II$ and $|\tau(\phi, s)| = 1$ for all $s \in Y$ then $\tau$ is deterministic.
	\end{enumerate}
\end{lemma}
\begin{proof}
	\begin{enumerate}
		\item Suppose that all $\tau_i$ are winning for the respective games, and consider any play $p_1 \ldots p_n = \tuple p \in P(G^M_X(\phi), II:\tau)$. Then $p_2$ is of the form $(\psi_i, s)$ for some $i \in 1 \ldots t$ and some $s \in Y_i$; and therefore, $p_2 \ldots p_n \in P(G^M_{Y_i}(\psi), II:\tau_i)$. But $\tau_i$ is winning, and hence $p_n$ is a winning position for Player $II$, as required. 
		\item Suppose that all $\tau_i$ are uniform, and consider any play $p_1 \ldots p_n = \tuple p \in P(G^M_X(\phi), II: \tau)$: then, once again, $p_2 \ldots p_n \in P(G^M_{Y_i}(\psi_i), II:\tau_i)$ for some $i$. \\

			Suppose now that $p_n$ is $(\tuple t_1 \subseteq \tuple t_2, s)$: since $\tau_i$ is uniform, there exists another play $q_1 \ldots q_m = \tuple q \in P(G^M_{Y_i}(\psi_i), II:\tau_i)$ such that $q_m = (\tuple t_1 \subseteq \tuple t_2, s')$ for the same instance of the inclusion atom and 
			\[
				\tuple t_1 \langle s\rangle = \tuple t_2 \langle s'\rangle. 
			\]
			Finally, $\tuple q$ is contained in a play of $(G^M_X(\phi), II:\tau)$ and hence the uniformity condition is respected for $\tau$. \\

			Suppose instead that $p_n$ is $(\tuple t_1 ~|~ \tuple t_2, s)$, and consider any other play $p'_1 \ldots p'_{n'}$ of $P(G^M_X(\phi), II:\tau)$ such that $p'_{n'}$ is $(\tuple t_1 ~|~ \tuple t_2, s')$ for the same instance of $\tuple t_1 ~|~ \tuple t_2$. Now, since the same exclusion atom is reached, it must be the case that $p'_2 \ldots p'_{n'}$ is in $P(G^M_Y(\psi_i), II:\tau_i)$ too, for the same $i$; but then, since $\tau_i$ is uniform, $\tuple t_1 \langle s\rangle \not = \tuple t_2 \langle s'\rangle$, as required.
	\item Let $p$ be any position in $G^M_X(\phi)$ such that $T(p) = II$. Then $p$ corresponds to a subformula of some $\psi_i$, and hence $|\tau(p)| = |\tau_i(p)| = 1$.
	\item Let $p$ be any position in $G^M_X(\phi)$ such that $T(p) = I$. If $p$ is $(\phi, s)$ for some $s \in Y$, then $|\tau(p)| = 1$ by hypothesis; and otherwise, $p$ corresponds to a subformula of some $\psi_i$, and as in the previous case $|\tau(p)| = |\tau_i(p)| = 1$.
	\end{enumerate}
\end{proof}

Finally, the connection between semantic games and team semantics is given by the following theorem:
\begin{theo}
	\label{theo:game_team_lax}
	Let $M$ be a first order model, let $\phi$ be an inclusion logic formula over the signature of $M$ and let $X$ be a team over $M$ whose domain contains all free variables of $\phi$. Then Player $II$ has a uniform winning strategy in $G^M_X(\phi)$ if and only if $M \models_X \phi$ (with respect to the lax semantics).
\end{theo}
\begin{proof}
	The proof is by structural induction on $\phi$. 
	\begin{enumerate}
		\item If $\phi$ is a first order literal then the only strategy available to $II$ in $G^M_X(\phi)$ is the empty one. This strategy is always uniform, and the plays which follow it are of the form $\tuple p = p_1 = (\phi, s)$, where $s$ ranges over $X$. Such a play is \emph{winning} for $II$ if and only if $M \models_s \phi$ in the usual first-order sense; and hence, the strategy is winning for $II$ if and only if $M \models_s \phi$ for all $s \in X$, that is, if and only if $M \models_X \phi$. 
		\item If $\phi$ is an inclusion atom $\tuple t_1 \subseteq \tuple t_2$ then, again, the only strategy available to Player $II$ is the empty one and the plays which follow it are those of the form $\tuple p = p_1 = (\tuple t_1 \subseteq \tuple t_2, s)$ for some $s \in X$. 

			By the definition of the winning positions of $G^M_X(\phi)$, this strategy is winning; hence, it only remains to check whether it is uniform. 

			Now, in order for the strategy to be uniform it must be the case that for all plays $\tuple p = p_1 = (\tuple t_1 \subseteq \tuple t_2, s)$ where $s \in X$ there exists a play $\tuple q = q_1 = (\tuple t_1 \subseteq \tuple t_2, s')$, again for $s' \in X$, such that $\tuple t_2 \langle s' \rangle = \tuple t_1 \langle s \rangle$. But this can be the case if and only if $\forall s \in X \exists s' \in X \mbox{ s.t. } \tuple t_1 \langle s\rangle = \tuple t_2\langle s'\rangle$, that is, if and only if $M \models_X \tuple t_1 \subseteq \tuple t_2$.
		\item If $\phi$ is an exclusion atom $\tuple t_1 ~|~ \tuple t_2$, the only strategy for $II$ in $G^M_X(\phi)$ is, once again, the empty one. This strategy is always winning, and it is uniform if and only if for all plays $\tuple p = p_1 = (\tuple t_1 ~|~ \tuple t_2, s)$ and $\tuple q = q_1 = (\tuple t_1 ~|~ \tuple t_2, s')$ (for $s, s' \in X$) it holds that $\tuple t_1 \langle s \rangle \not = \tuple t_2 \langle s'\rangle$.

			But this is the case if and only if $M \models_X \tuple t_1 ~|~ \tuple t_2$, as required. 

		\item If $\phi$ is a disjunction $\psi \vee \theta$, suppose that $\tau$ is a uniform winning strategy for $II$ in $G^M_X(\psi \vee \theta)$. Then define the teams $Y, Z \subseteq X$ as follows:
			\begin{align*}
				&Y = \{s \in X : (\psi, s) \in \tau(\psi \vee \theta, s)\};\\
				&Z = \{s \in X : (\theta, s) \in \tau(\psi \vee \theta, s)\}.
			\end{align*}
			Then $Y \cup Z = X$: indeed, for all $s \in X$ it must be the case that $\emptyset \not = \tau(\psi \vee \theta, s) \subsetneq \{(\psi, s), (\theta, s)\}$. Furthermore, $Y \cap Z = \emptyset$.\\

			Now consider the following two strategies for $II$ in $G^M_Y(\psi)$ and $G^M_Z(\theta)$ respectively: 
			\begin{itemize}
				\item $\tau_1(p) = \tau(p)$ for all positions $p$ of $G^M_Y(\psi)$; 
				\item $\tau_2(p) = \tau(p)$ for all positions $p$ of $G^M_Z(\theta)$.
			\end{itemize}
			Since all positions of $G^M_Y(\psi)$ and of $G^M_Z(\theta)$ are also positions of $G^M_X(\psi \vee \theta)$, $\tau_1$ and $\tau_2$ are well-defined.\\

			Furthermore, $(Y, \psi, \tau_1) \leq_M (X, \phi, \tau)$ and $(Z, \psi, \tau_2) \leq_M (X, \phi, \tau)$; therefore, by Lemma \ref{lemma:MSucc_Sub}, $\tau_1$ and $\tau_2$ are uniform and winning for $G^M_Y(\psi)$ and $G^M_Z(\theta)$. By induction hypothesis, this implies that $M \models_Y \psi$ and $M \models_Z \theta$, and by the definition of the semantics for disjunction, this implies that $M \models_X \psi \vee \theta$. 

			Conversely, suppose that $M \models_X \psi \vee \theta$: then, by definition, there exist teams $Y$ and $Z$ such that $X = Y \cup Z$, $M \models_Y \psi$ and $M \models_Z \theta$. Then, by induction hypothesis, there exist uniform winning strategies $\tau_1 $ and $\tau_2$ for $II$ in $G^M_Y(\psi)$ and $G^M_Z(\theta)$ respectively. Then define the strategy $\tau$ for $II$ in $G^M_X(\psi \vee \theta)$ as follows:
			\begin{itemize}
				\item $\tau(\psi \vee \theta, s) = \left\{
					\begin{array}{l l}
						\{(\psi, s)\} & \mbox{ if } s \in Y \backslash Z;\\
						\{(\theta, s)\} & \mbox{ if } s \in Z \backslash Y;\\
						\{(\psi, s), (\theta, s)\} & \mbox{ if } s \in Y \cap Z;
					\end{array}
					\right.$
				\item If $p$ is $(\chi, s)$ for some $s$ and some formula $\chi$ contained in $\psi$, then $\tau(p) = \tau_1(p)$;
				\item If $p$ is $(\chi, s)$ for some $s$ and some $\chi$ contained in $\theta$, then $\tau(p) = \tau_2(p)$.
			\end{itemize}
			Then, by construction, we have that $(Y, \psi, \tau_1), (Z, \theta, \tau_2) \leq_M (X, \psi \vee \theta, \tau)$; furthermore, $\psi$ and $\theta$ are all the immediate subformulas of $\psi \vee \theta$, and $\tau_1$ and $\tau_2$ are winning and uniform by hypothesis. Therefore, by Lemma \ref{lemma:MSucc_Sup}, $\tau$ is a uniform winning strategy for $G^M_X(\psi \vee \theta)$, as required.

\item If $\phi$ is $\psi \wedge \theta$, suppose again that $\tau$ is a uniform winning strategy for $II$ in $G^M_X(\psi \wedge \theta)$. Then consider the two strategies for $II$ in $G^M_X(\psi)$ and $G^M_Z(\theta)$, respectively, defined as 
			\begin{itemize}
				\item $\tau_1(p) = \tau(p)$ for all positions $p$ of $G^M_X(\psi)$; 
				\item $\tau_2(p) = \tau(p)$ for all positions $p$ of $G^M_X(\theta)$.
			\end{itemize}
			Then $(X, \psi, \tau_1), (X, \theta, \tau_2) \leq_M (X, \psi \wedge \theta, \tau)$, and therefore by Lemma \ref{lemma:MSucc_Sub} $\psi$ and $\theta$ are uniform winning strategies. Hence, by induction hypothesis, $M \models_X \psi$ and $M \models_X \theta$, and therefore $M \models_X \psi \wedge \theta$.
			
			Conversely, suppose that $M \models_X \psi \wedge \theta$. Then $M \models_X \psi$ and $M \models_X \theta$, and therefore $II$ has uniform winning strategies $\tau_1$ and $\tau_2$ for $G^M_X(\psi)$ and $G^M_X(\theta)$ respectively. Now define the strategy $\tau$ for $II$ in $G^M_X(\psi \wedge \theta)$ as follows:
			\[
			\mbox{for all } s \in X, \tau(\chi, s) = \left\{\begin{array}{l l}
				\tau_1(\chi, s) & \mbox{ if } \chi \mbox{ is contained in } \psi;\\
				\tau_2(\chi, s) & \mbox{ if } \chi \mbox{ is contained in } \theta.
			\end{array}
				\right.
			\]
			Then $(X, \psi, \tau_1), (X, \theta, \tau_2) \leq_M (X, \psi \wedge \theta, \tau)$ and $\psi, \theta$ are all immediate subformulas of $\psi \wedge \theta$; hence, by Lemma \ref{lemma:MSucc_Sup}, $\tau$ is a uniform winning strategy for $II$ in $G^M_X(\phi)$, as required.
		\item If $\phi$ is $\exists x \psi$, suppose that $\tau$ is a uniform winning strategy for $II$ in $G^M_X(\exists x \psi)$. Then define the function $H : X \rightarrow \mathcal P(\dom(M)) \backslash \emptyset$ as $H(s) = \{m \in M: (\psi, s[m/x]) \in \tau(\exists x \psi, s)\}$ and consider the following strategy $\tau'$ for $II$ in $G^M_{X[H/x]}(\psi)$: 
			\[
				\tau'(p) = \tau(p) \mbox{ for all suitable } p. 
			\]
			$\tau'$ is well-defined, because any position of $G^M_{X[H/x]}(\psi)$ is also a possible position of $G^M_X(\exists x \psi)$. Furthermore, $(X[H/x], \psi, \tau') \leq_M (X, \exists x \psi, \tau)$, and therefore $\tau'$ is a uniform winning strategy for $II$ in $G^M_{X[H/x]}(\psi)$. By induction hypothesis, this implies that $M \models_{X[H/x]} \psi$, and hence that $M \models_X \exists x \psi$. \\

			Conversely, suppose that $M \models_X \exists x \psi$; then, there exists a function $H$ such that $M \models_{X[H/x]} \psi$. By induction hypothesis, this means that there exists a winning strategy $\tau'$ for $II$ in $G^M_{X[H/x]}(\psi)$. Now consider the following strategy $\tau$ for $II$ in $G^M_{X}(\exists x \psi)$: 
			\begin{align*}
				&\tau(\exists x \psi, s) = \{(\psi, s[m/x]) : m \in H(s)\};\\
				&\tau(\theta, s) = \tau'(\tau, s) \mbox{ for all } \tau \mbox{ contained in } \psi \mbox{ and all } s.
			\end{align*}
			Then $(X[H/x], \psi, \tau') \leq_M (X, \exists x \psi, \tau)$, and $\psi$ is the only direct subformula of $\exists x \psi$; hence, $\tau$ is uniform and winning, as required. 
		\item If $\phi$ is $\forall x \psi$, suppose that $\tau$ is a uniform winning strategy for $II$ in $G^M_X(\forall x \psi)$. Then consider the strategy $\tau'$ for $II$ in $G^M_{X[M/x]}(\psi)$ given by 
			\begin{align*}
				&\tau'(\theta, s) = \tau(\theta, s) \mbox{ for all } \theta \mbox{ contained in } \psi \mbox{ and all } s.
			\end{align*}
			Then $(X[M/x], \psi, \tau') \leq_M (X, \forall x \psi, \tau)$, and hence $\tau'$ is uniform and winning. By induction hypothesis, this means that $M \models_{X[M/x]} \psi$, and hence that $M \models_X \forall x \psi$.\\

			Conversely, suppose that $M \models_X \forall x \psi$. Then $M \models_{X[M/x]} \psi$, and hence there exists a uniform winning strategy $\tau'$ for $II$ in $G^M_{X[M/x]}(\psi)$. Then consider the strategy $\tau$ for $II$ in $G^M_X(\psi)$ given by 
			\begin{align*}
				&\tau(\theta, s) = \tau'(\theta, s) \mbox{ for all } \theta \mbox{ contained in } \psi \mbox{ and all } s.
			\end{align*}
			This strategy is well-defined, since the first move of $G^M_X(\forall x \psi)$ is Player $I$'s; furthermore, 
	\[
		(X[M/x], \psi, \tau') \leq_M (X, \forall x \psi, \tau)
	\] 
and therefore $\tau$ is uniform and winning, as required.
	\end{enumerate}
\end{proof}
Hence, we have a game theoretic semantics which is equivalent to the lax team semantics for inclusion/exclusion logic; and of course, the game theoretic semantics for inclusion and exclusion logic are simply the restrictions of this semantics to the corresponding languages. As was argued previously, the strict team semantics for disjunction and existential quantification is somewhat less natural when it comes to inclusion logic or I/E logic. However, there exists a link between strict team semantics and \emph{deterministic} strategies:
\begin{theo}
	\label{theo:game_team_strict}
	Let $M$ be a first order model, let $\phi$ be an inclusion logic formula over the signature of $M$ and let $X$ be a team over $M$ whose domain contains all free variables of $\phi$. Then Player $II$ has a uniform, \emph{deterministic} winning strategy in $G^M_X(\phi)$ if and only if $M \models_X \phi$ (with respect to the \emph{strict} semantics).	
\end{theo}
\begin{proof}
	The proof is by structural induction over $\phi$, and it runs exactly as for the lax case. The only differences occur in the cases of disjunction and existential quantification, in which the determinism of the strategies poses a restriction on the choices available to Player $II$ and for which the proof runs as follows: 
	\begin{itemize}
		\item If $\phi$ is a disjunction $\psi \vee \theta$, suppose that $\tau$ is a uniform, deterministic winning strategy for $II$ in $G^M_X(\psi \vee \theta)$. Then define the teams $Y, Z \subseteq X$ as follows:
			\begin{align*}
				&Y = \{s \in X : \tau(\psi \vee \theta, s) = \{(\psi, s)\}\};\\
				&Z = \{s \in X : \tau(\psi \vee \theta, s) = \{(\theta, s)\}\}.
			\end{align*}
			Then $Y \cup Z = X$: indeed, for all $s \in X$ it must be the case that $\emptyset \not = \tau(\psi \vee \theta, s) \subseteq \{(\psi, s), (\theta, s)\}$, and hence $s$ is in $Y$ or in $Z$ (or in both). Furthermore, $Y \cap Z = \emptyset$.\\

			Now consider the following two strategies for $II$ in $G^M_Y(\psi)$ and $G^M_Z(\theta)$ respectively: 
			\begin{itemize}
				\item $\tau_1(p) = \tau(p)$ for all positions $p$ of $G^M_Y(\psi)$; 
				\item $\tau_2(p) = \tau(p)$ for all positions $p$ of $G^M_Z(\theta)$.
			\end{itemize}
			Since all positions of $G^M_Y(\psi)$ and of $G^M_Z(\theta)$ are also positions of \\$G^M_X(\psi \vee \theta)$, $\tau_1$ and $\tau_2$ are well-defined. Furthermore, they are deterministic, since $\tau$ is so, and $(Y, \psi, \tau_1), (Z, \theta, \tau_2) \leq_M (X, \phi, \tau)$; therefore, $\tau_1$ and $\tau_2$ are uniform and winning for $G^M_Y(\psi)$ and $G^M_Z(\theta)$. By induction hypothesis, this implies that $M \models_Y \psi$ and $M \models_Z \theta$; and by the definition of the (strict) semantics for disjunction, this implies that $M \models_X \psi \vee \theta$.\\

			Conversely, suppose that $M \models_X \psi \vee \theta$, according to the strict semantics: then, by definition, there exist teams $Y$ and $Z$ such that $X = Y \cup Z$, $Y \cap Z = \emptyset$, $M \models_Y \psi$ and $M \models_Z \theta$. Then, by induction hypothesis, there exist uniform, deterministic winning strategies $\tau_1$ and $\tau_2$ for $II$ in $G^M_Y(\psi)$ and $G^M_Z(\theta)$ respectively. Then define the strategy $\tau$ for $II$ in $G^M_X(\psi \vee \theta)$ as follows:
			\begin{itemize}
				\item $\tau(\psi \vee \theta, s) = \left\{
					\begin{array}{l l}
						\{(\psi, s)\} & \mbox{ if } s \in Y;\\
						\{(\theta, s)\} & \mbox{ if } s \in Z.\\
					\end{array}
					\right.$
				\item If $p$ is $(\chi, s)$ and $\chi$ is contained in $\psi$ then $\tau(p) = \tau_1(p)$;
				\item If $p$ is $(\chi, s)$ and $\chi$ is contained in $\theta$ then $\tau(p) = \tau_2(p)$.
			\end{itemize}
			Then, by construction, we have that 
			\[
				(Y, \psi, \tau_1), (Z, \theta, \tau_2) \leq_M (X, \psi \vee \theta, \tau);
			\]
			and furthermore, $\psi$ and $\theta$ are all the immediate subformulas of $\psi \vee \theta$, and $\tau_1$ and $\tau_2$ are winning and uniform by hypothesis. Therefore $\tau$ is a uniform, deterministic winning strategy for $G^M_X(\psi \vee \theta)$, as required.
				\item If $\phi$ is $\exists x \psi$, suppose that $\tau$ is a uniform, deterministic winning strategy for $II$ in $G^M_X(\exists x \psi)$. Then define the function $F : X \rightarrow \dom(M)$ so that, for every $s \in X$, $F(s)$ is the unique element $m$ of the model such that $\tau(\exists x \psi, s) = \{(\psi, s[m/x])\}$ and consider the following strategy $\tau'$ for $II$ in $G^M_{X[F/x]}(\psi)$: 
			\[
				\tau'(p) = \tau(p) \mbox{ for all suitable } p. 
			\]
			$\tau'$ is well-defined, because any position of $G^M_{X[F/x]}(\psi)$ is also a possible position of $G^M_X(\exists x \psi)$. Furthermore, $(Y[F/x], \psi, \tau') \leq_M (Y, \exists x \psi, \tau)$, and therefore $\tau'$ is a uniform, deterministic winning strategy for $II$ in $G^M_{X[F/x]}(\psi)$. By induction hypothesis, this implies that $M \models_{X[F/x]} \psi$, and hence that $M \models_X \exists x \psi$ (with respect to the strict semantics). \\

			Conversely, suppose that $M \models_X \exists x \psi$ according to the strict semantics; then, there exists a $F$ such that $M \models_{X[F/x]} \psi$. By induction hypothesis, this means that there exists a uniform, deterministic winning strategy $\tau'$ for $II$ in $G^M_{X[F/x]}(\psi)$. Now consider the following strategy $\tau$ for $II$ in $G^M_{X}(\exists x \psi)$: 
			\begin{align*}
				&\tau(\exists x \psi, s) = \{(\psi, s[F(s)/x])\};\\
				&\tau(\theta, s) = \tau'(\tau, s) \mbox{ for all } \tau \mbox{ contained in } \psi.
			\end{align*}
			Then $(X[F/x], \psi, \tau') \leq_M (X, \exists x \psi, \tau)$, and $\psi$ is the only direct subformula of $\exists x \psi$; hence, $\tau$ is uniform, deterministic and winning, as required. 
	\end{itemize}
\end{proof}
In \cite{forster06}, Thomas Forster considers the distinction between deterministic and nondeterministic strategies for the case of the logic of branching quantifiers and points out that, in the absence of the Axiom of Choice, different truth conditions are obtained for these two cases. In the same paper, he then suggests that
\begin{quote}
	Perhaps advocates of branching quantifier logics and their descendents will tell us which semantics \textbf{[that is, the deterministic or nondeterministic one]} they have in mind.
\end{quote}
Dependence logic, inclusion logic, inclusion/exclusion logic and independence logic can certainly be seen as descendents of branching quantifier logic, and the present work strongly suggests that the semantics that we ``have in mind'' is the nondeterministic one. As we just saw, the deterministic/nondeterministic distinction in game theoretic semantics corresponds precisely to the strict/lax distinction in team semantics; and indeed, as seen in Subsection \ref{subsect:laxstrict}, for dependence logic proper (which is expressively equivalent to branching quantifier logic), the lax and strict semantics are equivalent modulo the Axiom of Choice (Proposition \ref{propo:laxeqstrict}). 

But for inclusion logic and its extensions, we have that lax and strict (and, hence, nondeterministic and deterministic) semantics are not equivalent, even in the presence of the Axiom of Choice (Propositions \ref{propo:dislaxneqstr} and \ref{propo:exlaxneqstr}), and that only the lax one satisfies Locality in the sense of Theorem \ref{theo:DLloc} (see Proposition \ref{propo:strict_nonlocal} and Theorems \ref{theo:strict_local}, \ref{theo:ielocal} for the proof).

Furthermore, as stated before, Fredrik Engstr\"om showed in \cite{engstrom10} that the lax semantics for existential quantification arises naturally from his treatment of generalized quantifiers in dependence logic.

All of this, in the opinion of the author at least, makes a convincing case for the adoption of the nondeterministic semantics (or, in terms of team semantics, of the lax one) as the natural semantics for the study of logics of imperfect information, thus suggesting an answer to Thomas Forster's question. 
\section{Definability in I/E logic (and in independence logic)}
\label{sect:defin}
In \cite{kontinenv09}, Kontinen and V\"a\"an\"anen characterized the expressive power of dependence logic formulas (Theorem \ref{theo:DLform} here), and, in \cite{kontinennu09}, Kontinen and Nurmi used a similar technique to prove that a class of teams is definable in team logic (\cite{vaananen07b}) if and only if it is expressible in full second order logic.\\

In this section, I will attempt to find an analogous result for I/E logic (and hence, through Corollary \ref{coro:IL_eq_IE}, for independence logic). One direction of the intended result is straightforward: 
\begin{theo}
\label{theo:IE2Sigma11}
	Let $\phi(\tuple v)$ be a formula of I/E logic with free variables in $\tuple v$. Then there exists an existential second order logic formula $\Phi(A)$, where $A$ is a second order variable with arity $|\tuple v|$, such that 
	\[
		M \models_X \phi(\tuple v) \Leftrightarrow M \models \Phi(\rel_{\tuple v}(X))
	\]
	for all suitable models $M$ and teams $X$.
\end{theo}
\begin{proof}
	The proof is an unproblematic induction over the formula $\phi$, and follows closely the proof of the analogous results for dependence logic (\cite{vaananen07}) or independence logic (\cite{gradel10}). 
\end{proof}

The other direction, instead, requires some care:\footnote{The details of this proof are similar to the ones of \cite{kontinenv09} and \cite{kontinennu09}.}
\begin{theo}
\label{theo:Sigma112IL}
	Let $\Phi(A)$ be a formula in $\Sigma_1^1$ such that $\free(\Phi) = \{A\}$, and let $\tuple v$ be a tuple of distinct variables with $|\tuple v| = \arity(A)$. Then there exists an I/E logic formula $\phi(\tuple v)$ such that 
	\[
		M \models_X \phi(\tuple v) \Leftrightarrow M \models \Phi(\rel_{\tuple v}(X))
	\]
	for all suitable models $M$ and nonempty teams $X$.
\end{theo}
\begin{proof}
	It is easy to see that any $\Phi(A)$ as in our hypothesis is equivalent to the formula
	\[
		\Phi^*(A) = \exists B ( \forall \tuple x (A \tuple x \leftrightarrow B \tuple x) \wedge \Phi(B)),
	\]
	in which the variable $A$ occurs only in the conjunct $\forall \tuple x(A \tuple x \leftrightarrow B \tuple x)$. Then, as in \cite{kontinenv09}, it is possible to write $\Phi^*(A)$ in the form 
	\[
		\exists \tuple f~ \forall \tuple x \tuple y ( (A \tuple x \leftrightarrow f_1(\tuple x) = f_2(\tuple x)) \wedge \psi(\tuple x, \tuple y, \tuple f)),
	\]
	where $\tuple f = f_1 f_2 \ldots f_n$, $\psi(\tuple f, x, y)$ is a quantifier-free formula in which $A$ does not appear, and each $f_i$ occurs only as $f(\tuple w_i)$ for some fixed tuple of variables $\tuple w_i \subseteq \tuple x \tuple y$. 

	Now define the formula $\phi(\tuple v)$ as
	\[
		\forall \tuple x \tuple y ~\exists \tuple z \left( 
		\bigwedge_i =\!\!(\tuple w_i, z_i) \wedge ( ((\tuple v \subseteq \tuple x \wedge z_1 = z_2) \vee (\tuple v ~|~ \tuple x \wedge z_1 \not = z_2)) \wedge \psi'(\tuple x, \tuple y, \tuple z))\right),
	\]
	where $\psi'(\tuple x, \tuple y, \tuple z)$ is obtained from $\psi(\tuple x, \tuple y, \tuple f)$ by substituting each $f_i(\tuple w_i)$ with $z_i$, and the dependence atoms are used as shorthands for the corresponding expressions of I/E logic. 

	Now we have that $M \models_X \phi(\tuple v) \Leftrightarrow M \models \Phi^*(\rel_{\tuple v}(X))$:
	
	Indeed, suppose that $M \models_X \phi(\tuple v)$. Then, by construction, for each $i = 1 \ldots n$ there exists a function $F_i$, depending only on $\tuple w_i$, such that for $Y = X[M/\tuple x \tuple y][\tuple F / \tuple z]$ 
\[
	M \models_{Y} ((\tuple v \subseteq \tuple x \wedge z_1 = z_2) \vee (\tuple v ~|~ \tuple x \wedge z_1 \not = z_2)) \wedge \psi'(\tuple x, \tuple y, \tuple z).
\]

Therefore, we can split $Y$ into two subteams $Y_1$ and $Y_2$ such that $M \models_{Y_1} \tuple v \subseteq \tuple x \wedge z_1 = z_2$ and $M \models_{Y_2}  \tuple v ~|~ \tuple x \wedge z_1 \not = z_2$.

Now, for each $i$ define the function $f_i$ so that, for every tuple $\tuple m$ of the required arity, $f_i(\tuple m)$ corresponds to $F_i(s)$ for an arbitrary $s \in X[M/\tuple x \tuple y]$ with $s(\tuple w_i) = \tuple m$, and let $o$ be any assignment with domain $\tuple x \tuple y$. 

Thus, if we can prove that $M \models_o ((\rel_{\tuple v}(X)) \tuple x \leftrightarrow f_1(\tuple x) =  f_2 (\tuple x)) \wedge \psi(\tuple x, \tuple y, \tuple f)$ then the left-to-right direction of our proof is done. 

First of all, suppose that $M \models_o (\rel_{\tuple v}(X)) \tuple x$, that is, that $o(\tuple x) = \tuple m = s(\tuple v)$ for some $s \in X$. 

Then choose an arbitrary tuple of elements $\tuple r$ and consider the assignment $h = s[\tuple m/\tuple x][\tuple r/\tuple y][\tuple F/\tuple z] \in Y$. This $h$ cannot belong to $Y_2$, since $h(\tuple v) = s(\tuple v) = \tuple m = h(\tuple x)$, and therefore it is in $Y_1$ and $h(z_1) = h(z_2)$.

By the definition of the $f_i$, this implies that $f_1 (\tuple m) = f_2 (\tuple m)$, as required. \\

Analogously, suppose that $M, \not \models_o (\rel_{\tuple v}(X))\tuple x$, that is, that $o(\tuple x) = \tuple m \not = s(\tuple v)$ for all $s \in X$. 
Then pick an arbitrary such $s \in X$ and an arbitrary tuple of elements $\tuple r$, and consider the assignment
\[
	h = s[\tuple m/\tuple x][\tuple r/\tuple y][\tuple F/\tuple z] \in Y.
\]

If $h$ were in $Y_1$, there would exist an assignment $h' \in Y_1$ such that $h'(\tuple v) = h(\tuple x) = \tuple m$; but this is impossible, and therefore $h \in Y_2$. Hence $h(z_1) \not = h(z_2)$, and therefore $f_1(\tuple m) \not = f_2(\tuple m)$. 

Putting everything together, we just proved that
\[
	M \models_o R\tuple x \Leftrightarrow f_1(\tuple x) =  f_2 (\tuple x)
\]
for all assignments $o$ with domain $\tuple x \tuple y$, and we still need to verify that $M \models_o \psi(\tuple x, \tuple y, f)$ for all such $o$. 

But this is immediate: indeed, let $s$ be an arbitrary assignment of $X$, and construct the assignment
\[
	h = s[o(\tuple x \tuple y)/\tuple x \tuple y][\tuple F/\tuple z] \in X[M/\tuple x \tuple y][\tuple F/\tuple z].
\]

Then, since $M \models_{X[M/\tuple x \tuple y][\tuple F/\tuple z]} \psi'(\tuple x, \tuple y, \tuple z)$ and $\psi'(\tuple x, \tuple y, \tuple z)$ is first order, $M \models_{\{h\}} \psi'(\tuple x, \tuple y, \tuple z)$; but $\psi'(\tuple x, \tuple y, \tuple f(\tuple x \tuple y))$ is equivalent to $\psi(\tuple x, \tuple y, \tuple f)$ and $h(z_i) =  f(h(\tuple w_i)) = f(o(\tuple w_i))$, and therefore
\[
	M \models_o \psi(\tuple x, \tuple y, \tuple f)
\]
as required.
$\\$

Conversely, suppose that $M \models_s  (\rel_{\tuple v}(X))\tuple x \leftrightarrow (f_1(\tuple x) =  f_2 (\tuple x)) \wedge \psi(\tuple x, \tuple y, \tuple f)$ for all assignments $s$ with domain $\tuple x \tuple y$ and for some fixed choice of the tuple of functions $\tuple f$.

Then let $\tuple F$ be such that, for all assignments $h$ and for all $i$,
\[
	F_i(h) = f_i(h(\tuple w_i))
\]
and consider $Y = X[M/\tuple x \tuple y][F/\tuple z]$. 

Clearly, $Y$ satisfies the dependency conditions; furthermore, it satisfies $\psi'(\tuple x, \tuple y, \tuple z)$, because for every assignment $h \in Y$ and every $i \in 1 \ldots n$ we have that $h(z_i) = F_i(h) = f_i(h(\tuple w_i))$.

Finally, we can split $Y$ into two subteams $Y_1$ and $Y_2$ as follows:
\begin{align*}
& Y_1 = \{o \in Y : o(\tuple z_1) = o(\tuple z_2)\};\\
& Y_2 = \{o \in Y : o(\tuple z_1) \not = o(\tuple z_2)\}.
\end{align*}

It is then trivially true that $M \models_{Y_1} z_1 = z_2$ and $M \models_{Y_2} z_1 \not= z_2$, and all that is left to do is proving that $M \models_{Y_1} \tuple v \subseteq \tuple x$ and $M \models_{Y_2} \tuple v ~|~ \tuple x$.

As for the former, let $o \in Y_1$: then, since $o(z_1) = o(z_2)$, $f_1(o(\tuple x)) = f_2(o(\tuple x))$.

This implies that $o(\tuple x) \in \rel_{\tuple v}(X)$, and hence that there exists an assignment $s' \in X$ with $s'(\tuple v) = o(\tuple x)$. 

Now consider the assignment 
\[
	o' = s'[o(\tuple x \tuple y)/\tuple x \tuple y][\tuple F/\tuple z]:
\]
since in $Y$ the values of $\tuple z$ depend only on the values of $\tuple x \tuple y$ and since $o(z_1) = o(z_2)$, we have that $o'(z_1) = o'(z_2)$ and hence $o' \in Y_1$ too. But $o'(\tuple v) = s'(\tuple v) = o(\tuple x)$, and since $o$ was an arbitrary assignment of $Y_1$, this implies that $M \models_{Y_1} \tuple v \subseteq \tuple x$. 

Finally, suppose that $o \in Y_2$. Then, since $o(z_1) \not = o(z_2)$, we have that $f_1(o(\tuple x)) \not= f_2(o(\tuple x))$; and therefore, $o(\tuple x) \not \in \rel_{\tuple v}(X)$, that is, for all assignments $s \in X$ it holds that $s(\tuple v) \not = o(\tuple x)$. Then the same holds for all $o' \in Y_2$. 

This concludes the proof.
\end{proof}
Since by Corollary \ref{coro:IL_eq_IE} we already know independence logic and I/E logic have the same expressive power, this has the following corollary: 
\begin{coro}
	\label{coro:ILform}
	Let $\Phi(A)$ be an existential second order formula with $\free(\Phi) = A$, and let $\tuple v$ be any set of variables such that $|\tuple v| = \arity(A)$. Then there exists an independence logic formula $\phi(\tuple v)$ such that 
	\[
		M \models_X \phi(\tuple v) \Leftrightarrow M \models \Phi(\rel_{\tuple v}(X)) 
	\]
	for all suitable models $M$ and teams $X$. 
\end{coro}
Finally, by Fagin's Theorem (\cite{fagin74}) this gives an answer to Gr\"adel and V\"a\"an\"anen's question: 
\begin{coro}
All NP properties of teams are expressible in independence logic.
\end{coro}

This result has far-reaching consequences. First of all, it implies that independence logic (or, equivalently, I/E logic) is the most expressive logic of imperfect information which only deals with existential second order properties. Extensions of independence logic can of course be defined; but unless they are capable of expressing some property which is not existential second order (as, for example, is the case for the intuitionistic dependence logic of \cite{yang10}, or for the $BID$ logic of \cite{abramsky08}), they will be expressively equivalent to independence logic proper.  As (Jouko V\"a\"an\"anen, private communication) pointed out, this means that independence logic is \emph{maximal} among the logics of imperfect information which always generate existential second order properties of teams. In particular, \emph{any} dependency condition which is expressible as an existential second order property over teams can be expressed in independence logic: and as we will see in the next section, this entails that such a logic is capable of expressing a great amount of the notions of dependency considered by database theorists.
\section{Equality generating dependencies, tuple generating dependencies and independence logic}
\label{sect:eqtupgen}
In Database Theory, two of the most general notions of dependence are \emph{tuple generating} and \emph{equality generating} dependencies. 

In brief, a \emph{tuple generating dependency} over a database relation $R$ is a sentence of the form 
\[
	\Delta(A) = \forall x_1 \ldots x_n (\phi(x_1 \ldots x_n) \rightarrow \exists z_1 \ldots z_k \psi(x_1 \ldots x_n, z_1 \ldots z_k))
\]
where $A$ is a second order variable with arity equal to the number of attributes of $R$.\footnote{In other words, if we consider $R$ as a relation in first order logic then $\arity(A) = \arity(R)$.} and $\phi$ and $\psi$ are conjunctions of atoms of the form $A \tuple t$ or $\tuple t_1 = \tuple t_2$ for some terms $\tuple t$, $\tuple t_1$ and $\tuple t_2$ in the empty vocabulary and with free variables in $x_1 \ldots x_n$.
  
An \emph{equality generating dependency} is defined much in the same way, except that $\psi$ is a single equality atom instead. \\

Then, given a domain of predication $M$, a relation $R$ is said to satisfy a (tuple-generating or equality-generating) dependency $\Delta$ if and only if $M \models \Delta(R)$ in the usual first order sense.

As an example of the expressive power of tuple-generating and equality-generating dependencies, let us observe that dependency atoms correspond to equality generating dependencies and that independence atoms correspond to tuple generating dependencies: indeed, for example, $M \models_X =\!\!(x, y)$ if and only if 
\[
M \models \forall x y_1 y_2 \tuple z_1 \tuple z_2 ((\rel(X))xy_1\tuple z_1 \wedge (\rel(X))x y_2 \tuple z_2 \rightarrow y_1 = y_2)
\]
where $|\tuple z_1| = |\tuple z_2| = |\dom(X) \backslash \{x,y\}|$, and $M \models_X \indep{x}{y}{z}$ if and only if 
\begin{align*}
M \models & \forall x y_1 y_2 z_1 z_2 \tuple w_1 \tuple w_2 ( ( (\rel(X)) x y_1 z_1 \tuple w_1 \wedge  (\rel(X)) x y_2 z_2 \tuple w_2) \rightarrow\\
& \rightarrow \exists \tuple w_3  (\rel(X)) x y_1 z_2 \tuple w_3).
\end{align*}

From the main result of the previous section, it is easy to see that I/E logic (and, as a consequence, independence logic) is capable to express all tuple and equality generating dependencies:
\begin{propo}
	Let $\Delta(A)$ be a tuple generating or equivalent generating dependency, and let $\tuple v$ be a tuple of distinct variables with $|\tuple v| = \arity(A)$. Then there exists an I/E logic (or independence logic) formula $\phi(\tuple v)$ such that 
	\[
		M \models_X \phi(\tuple v) \Leftrightarrow M \models \Delta(\rel_{\tuple v}(X))
	\]
	for all suitable models $M$ and all teams $X$ with $\tuple v \subseteq \dom(X)$.
\end{propo}
\begin{proof}
	$\Delta(A)$ is definable by a first order formula, and hence by Theorem \ref{theo:Sigma112IL} it is expressible in I/E logic (and therefore by independence logic too, by Corollary \ref{coro:IL_eq_IE}).
\end{proof}
Hence, many of the properties which are discussed in the context of Database Theory can be expressed through independence logic. The vast expressive power of this formalism comes with a very high computational cost, of course; but it is the hope of the author that the result of this work may provide a justification to the study of this logic (and, more in general, of logics of imperfect information) as a general theoretic framework for reasoning about knowledge bases. 
\section{Acknowledgements}
The author wishes to thank Jouko V\"a\"an\"anen for many valuable insights and suggestions. Furthermore, he thanks Erich Gr\"adel for having suggested a better notation for inclusion and exclusion dependencies, and Allen Mann for having mentioned Thomas Forster's paper. Finally, the author thankfully acknowledges the support of the EUROCORES LogICCC LINT programme.



\bibliographystyle{plain}
\bibliography{biblio}







\end{document}